\documentclass[reqno]{amsart}

\usepackage{enumitem}
\usepackage{hyperref}

\newcommand*{\mailto}[1]{\href{mailto:#1}{\nolinkurl{#1}}}

\newtheorem{theorem}{Theorem}[section]
\newtheorem{definition}[theorem]{Definition}
\newtheorem{lemma}[theorem]{Lemma}
\newtheorem{proposition}[theorem]{Proposition}
\newtheorem{corollary}[theorem]{Corollary}

\newtheorem{remark}[theorem]{Remark}

\newcommand{\B}{{\mathcal B}}

\newcommand{\R}{{\mathbb R}}
\newcommand{\N}{{\mathbb N}}

\newcommand{\C}{{\mathbb C}}

\newcommand{\E}{\mathrm{e}}

\newcommand{\sgn}{\mathrm{sgn}}

\newcommand{\re}{\mathrm{Re}}
\newcommand{\supp}{\mathrm{supp}}

\newcommand{\OO}{\mathcal{O}}
\newcommand{\oo}{o}
\newcommand{\cI}{\mathcal{I}}

\newcommand{\dip}{\upsilon}

\renewcommand{\labelenumi}{(\roman{enumi})}
\renewcommand{\theenumi}{\roman{enumi}}
\numberwithin{equation}{section}

\begin{document}

\title[An isospectral problem for multi-peakon solutions]{An isospectral problem for global conservative multi-peakon solutions of the Camassa--Holm equation}

\author[J.\ Eckhardt]{Jonathan Eckhardt}
\address{Institut Mittag-Leffler\\
Aurav\"agen 17\\ SE-182 60 Djursholm\\ Sweden}
\email{\mailto{jonathaneckhardt@aon.at}}

\author[A.\ Kostenko]{Aleksey Kostenko}
\address{Faculty of Mathematics\\ University of Vienna\\
Nordbergstrasse 15\\ 1090 Wien\\ Austria}
\email{\mailto{duzer80@gmail.com}; \mailto{Oleksiy.Kostenko@univie.ac.at}}

\thanks{\href{http://dx.doi.org/10.1007/s00220-014-1905-4}{Comm.\ Math.\ Phys.\ {\bf 329} (2014), no.~3, 893--918}}
\thanks{{\it Research supported by the Austrian Science Fund (FWF) under Grants No.\ Y330 and M1309 as well as by the AXA Mittag-Leffler Fellowship Project, funded by the AXA Research Fund}}


\begin{abstract}
 We introduce a generalized isospectral problem for global conservative multi-peakon solutions of the Camassa--Holm equation.  
 Utilizing the solution of the indefinite moment problem given by M.\ G.\ Krein and H.\ Langer, we show that the conservative Camassa--Holm equation is integrable by the inverse spectral transform in the multi-peakon case.
\end{abstract}

\maketitle

\section{Introduction}
 
 Over the last two decades, a lot of work has been devoted to the Cauchy problem for the Camassa--Holm equation, a nonlinear wave equation, given by 
  \begin{equation}\label{eqnCH}
   u_{t} -u_{xxt}  = 2u_x u_{xx} - 3uu_x + u u_{xxx},\quad  u|_{t=t_0} = u_0.
  \end{equation}
  For further information, we only refer to a brief selection of articles \cite{caho93, co01, coes98, como00, cost00, mc03, mc04}. 
  The Camassa--Holm equation first appeared as an abstract bi-Hamiltonian partial differential equation in an article of Fokas and Fuchssteiner \cite{fofu81}.
  However, it did not receive much attention until Camassa and Holm \cite{caho93} derived it as a nonlinear wave equation which models unidirectional wave propagation on shallow water. 
  Regarding the hydrodynamical relevance of this equation, let us also mention the more recent articles \cite{cola09, io07, jo02}, containing further information. 
  
  One of the most eminent properties of the Camassa--Holm equation lies in the fact that it is formally integrable in the sense that there is an associated Lax pair.
  The isospectral problem of this Lax pair turned out to be the weighted Sturm--Liouville equation  
 \begin{align}\label{eqnISP}
  -f''(x) + \frac{1}{4} f(x) = z\, \omega(x,t) f(x), \quad x\in\R,
 \end{align} 
 where $\omega=u - u_{xx}$ and $z\in\C$ is a complex spectral parameter. 
 Of course, (inverse) spectral theory for this Sturm--Liouville problem is of peculiar interest for solving the Cauchy problem of the Camassa--Holm equation; \cite{besasz00, be04, bebrwe08, bebrwe12, co01, cogeiv06, IsospecCH}. 
 
 A particular kind of solutions of the Camassa--Holm equation are the so-called {\em multi-peakon solutions}. 
 These are solutions of the form  
 \begin{align}\label{eqnMP}
  u(x,t) = \sum_{n=1}^N p_n(t)\, \E^{-|x-q_n(t)|}, 
 \end{align}
 where the functions on the right-hand side satisfy the following nonlinear system of ordinary differential equations: 
 \begin{align}\label{eqnMPsys}
  q_n' & = \sum_{k=1}^N p_k\, \E^{-|q_n-q_k|}, &
  p_n' & = \sum_{k=1}^N p_n p_k\, \sgn(q_n - q_k)\, \E^{-|q_n-q_k|}.
 \end{align}
 Note that the system \eqref{eqnMPsys} is Hamiltonian, that is,  
 \begin{align}\label{eqnMPsysH}
  \frac{d q_n}{dt} = \frac{\partial H(p,q)}{\partial p_n}, \qquad 
  \frac{d p_n}{dt} = - \frac{\partial H(p,q)}{\partial q_n}, 
 \end{align}
 with the Hamiltonian given by 
 \begin{align}\label{eq:H}
  H(p,q)= \frac{1}{2} \sum_{n,k=1}^N p_n p_k \, \E^{-|q_n-q_k|}=\frac{1}{4} \|u\|^2_{H^1(\R)}.
 \end{align}
 Since multi-peakon solutions \eqref{eqnMP} obviously have discontinuous first derivatives, they have to be interpreted as suitable weak solutions \cite{besasz00, coes98, como00, hora06}. 
 
 Note that the right-hand side in~\eqref{eqnMPsys} is not Lipschitz if $q_n-q_k$ is close to zero and hence in this case, one cannot get existence and uniqueness of solutions of \eqref{eqnMPsys} by using the standard arguments. 
 However, if we know in advance that all the positions stay distinct, then the right-hand side in \eqref{eqnMPsys} becomes Lipschitz and thus the Picard theorem applies. 
 In fact, the behavior of multi-peakon solutions crucially depends on whether all the heights $p_n$ of the single peaks are of the same sign or not.    
 In the first case, all the positions $q_n$ of the peaks stay distinct, move in the same direction and the system~\eqref{eqnMPsys} allows a unique global solution \cite{coes98, como00, hora06}. 
 Otherwise, some of the positions $q_n$ of the peaks will collide eventually, which causes the corresponding heights $p_n$ to blow up in finite time \cite{caho93}.
 All this happens in such a way that the solution $u$ in~\eqref{eqnMP} stays uniformly bounded in $H^1(\R)$ but its derivative develops a singularity at the points where two peaks collide. 
 
 The wave-breaking process described above is not only a peculiarity of multi-peakon solutions, but also occurs for smooth solutions \cite{co00, coes98, coes98b, mc04}.  
 In fact, the criteria for a blow-up to happen do not depend on smoothness of the initial data but only on the sign changes of $\omega$.
 For example, solutions are known only to blow up if the corresponding quantity $\omega$ is not of one sign.  
 In this case, the blow-up again happens in such a way that the solution $u$ stays uniformly bounded in $H^1(\R)$ but its derivative becomes unbounded, resembling wave-breaking. 

 However, it turned out that the encountered blow-up is not too severe and that it is always possible to continue weak solutions beyond wave-breaking \cite{xizh00}.
 In order to end up with unique continuations, one has to impose additional constraints on the solutions. 
 For example, if one requires the energy of the solutions to be conserved, one is led to the notion of {\em global conservative solutions} \cite{brco07, hora07, hora07b}.
 Although there are further possibilities to guarantee uniqueness \cite{brco07a, hora08, hora09}, the conservative case is the suitable one for our purposes.  
 For the corresponding Cauchy problem to be well-posed, it is necessary to introduce an additional quantity, which measures the energy density of the solution (as done recently in \cite{brco07, hora07}).   
 Following \cite{hora07}, {\em a global conservative solution} consists of a pair $(u,\mu)$, where $\mu$ is a non-negative Borel measure with absolutely continuous part determined by $u$ via 
 \begin{align}
  \mu_{\text{ac}}(B,t) = \int_B |u(x,t)|^2 + |u_x(x,t)|^2 dx, \quad t\in\R,
 \end{align}
 for each Borel set $B\in\B(\R)$. 
 Within this picture, blow-up of solutions corresponds to concentration of energy (measured by $\mu$) to sets of Lebesgue measure zero. 
 
 The notation we use in this article is based on \cite{hora07b}, where a detailed description of global conservative multi-peakon solutions was given (in Section~\ref{sec:mp} we will provide a brief review of this notion).  
 In particular, there \cite[Section~4]{hora07b}, the potential blow-up of the system~\eqref{eqnMPsys} was circumvented by reformulating it in Lagrangian coordinates.   
 The newly obtained system \cite[(4.1)]{hora07b} of ordinary differential equations remains globally well-defined.

 For the special case of multi-peakon solutions, the weight $\omega$ in~\eqref{eqnISP} is always a finite sum of weighted Dirac measures.
 The corresponding spectral problem~\eqref{eqnISP} is equivalent to the one for an indefinite Krein--Stieltjes string \cite[\S 13]{kakr74}.
 This connection and the solution of the corresponding inverse problem due to Krein (employing Stieltjes theory of continued fractions) has successfully been employed by Beals, Sattinger and Szmigielski \cite{besasz98, besasz00} in order to study multi-peakon solutions (in the sense of~\eqref{eqnMPsys}). 
  In particular, they noticed that in the indefinite case, the inverse problem is not always solvable within the class of spectral problems~\eqref{eqnISP}, which directly corresponds to the fact that the system~\eqref{eqnMPsys} may blow up. 
 It is the purpose of the present article to introduce a generalized isospectral problem for global conservative multi-peakon solutions of the Camassa--Holm equation. 
 Of course, an eligible spectral problem also has to incorporate the singular part of $\mu$ in some way and indeed, it turns out that the appropriate generalized spectral problem is given by 
 \begin{align}\label{eqnGISO}
  -f''(x) + \frac{1}{4} f(x) = z\, \omega(x,t) f(x) + z^2 \dip(x,t) f(x), \quad x\in\R, 
 \end{align}
 where $\dip(\,\cdot\,,t)$ denotes the singular part of $\mu(\,\cdot\,,t)$ and $z\in\C$ is a complex spectral parameter. 
 The idea for considering this particular spectral problem goes back to work of Krein and Langer \cite{krla79, krla80} (see also \cite{de97, lawi98}) on the indefinite moment problem and generalized strings which carry not only finitely many negative point masses but also dipoles. 
 
  Let us now briefly outline the content of the present article. 
  Necessary facts on global conservative multi-peakon solutions are collected in the preliminary Section~\ref{sec:mp}.
  In the following section, we will discuss the basic properties of the generalized spectral problem~\eqref{eqnGISO}. 
  Section~\ref{secIP} provides a solution of the corresponding inverse spectral problem which is essentially due to Krein and Langer \cite{krla79, krla80}.  
  Of course, the solvability of this inverse problem resembles the fact that the conservative Camassa--Holm equation has global solutions. 
  In Section~\ref{secGCMPS}, we will show that our generalized spectral problem indeed serves as an isospectral problem for the conservative Camassa--Holm equation in the multi-peakon case.
  This is done by deriving the time evolution for the spectral quantities associated with~\eqref{eqnGISO}. 
  As a consequence, we immediately obtain conserved quantities for the conservative Camassa--Holm flow.  
  In particular, {\em the conservative Camassa--Holm equation turns out to be a completely integrable Hamiltonian system in the multi-peakon case}.  
  Finally, we demonstrate our findings in Appendix~\ref{App}, using the example of a general global conservative two-peakon solution of the Camassa--Holm equation.

\section{Global conservative multi-peakon solutions}\label{sec:mp}

Instead of providing details about general global conservative solutions of the Camassa--Holm equation \cite{brco07, hora07}, we will review this notion only in the multi-peakon context. 
Therefore, we will closely follow the notation employed in~\cite{hora07b}, where a detailed description of global conservative multi-peakon solutions was given.  

\begin{definition}\label{def:mp}
 A global conservative solution $(u,\mu)$ of the Camassa--Holm equation is said to be {\em a multi-peakon solution} if for some $t_0\in\R$, the measure $\mu(\,\cdot\,,t_0)$ is absolutely continuous and  
 \begin{align}\label{eqnIVGCMP}
  u(x,t_0) =  \sum_{n=1}^{N} p_n(t_0) \, \E^{-|x-q_n(t_0)|}, \quad x\in\R,
 \end{align}
 for some $N\in\N_0$ and $p_n(t_0)$, $q_n(t_0)\in\R$ for $n=1,\ldots,N$. 
\end{definition}

More generally, one could also allow the singular part of the measure $\mu(\,\cdot\,,t_0)$ to be supported on a finite set in this definition (and still end up with the same notion).  
%  this is not explicitely proved in \cite{hora07b}, given the results of this paper this can be shown as follows:
%  Associate spectral quantities with $u(\,\cdot\,,t_0)$ and $\mu(\,\cdot\,,t_0)$, evolve them in time according to \eqref{eqnTE}, then for some time $t_1\in\R$ the inverse problem has a solution $\omega_1$, $\dip_1$ with $\dip_1=0$
%  Define $u_1$ and $\mu_1$ in the obvious way and denote with $(\tilde{u},\tilde{\mu})$ the global conservative solution with $\tilde{u}(\,\cdot\,,t_1) = u_1$ and $\tilde{\mu}(\,\cdot\,,t_1) = \mu_1$.
%  According to Theorem~\ref{thmTE}, the spectral quantities associated with $\tilde{u}(\,\cdot\,,t_0)$ and $\tilde{\mu}(\,\cdot\,,t_0)$ are equal to those of $u(\,\cdot\,,t_0)$ and $\mu(\,\cdot\,,t_0)$. 
%  In view of inverse uniqueness, $(u,\mu)$ is actually equal to $(\tilde{u},\tilde{\mu})$ which is a multi-peakon solution in the former sense.
 However, for the following description of global conservative multi-peakon solutions, it is much more convenient to assume the measure $\mu(\,\cdot\,,t_0)$ to be absolutely continuous. 
 Furthermore, for definiteness, we will assume that the heights $p_n(t_0)$, $n=1,\ldots,N$ of the single peaks are non-zero and that their positions are strictly increasing,
 \begin{align}\label{eq:xneqx}
  -\infty < q_1(t_0) < \cdots < q_N(t_0) < \infty. 
 \end{align}
 It is a result of \cite[Section~3]{hora07b} that in this multi-peakon case, a global conservative solution $(u,\mu)$ will be of the form~\eqref{eqnMP} for all times $t\in\R$.   
 More precisely, if for each $n\in\lbrace 1,\ldots,N\rbrace$, the function $q_n$ on $\R$ denotes the characteristic through the point $q_n(t_0)$ at time $t_0$, that is, 
 \begin{align}
  q_n'(t) & = u(q_n(t),t), \quad t\in\R,  
 \end{align} 
 (cf.\ \eqref{eqnMPsys}), then the function $u$ may be written as
 \begin{align}\label{eqnuGCMP}
  u(x,t) = \sum_{n=1}^N p_n(t) \, \E^{-|x-q_n(t)|}, \quad x,\, t\in\R, 
 \end{align}
 for some $p_n(t)\in\R$, $n=1,\ldots,N$.
 Hereby note that for each $n\in\lbrace 1,\ldots,N\rbrace$, the height $p_n(t)$ is uniquely determined by $u(\,\cdot\,,t)$ unless the corresponding characteristic $q_n$ coincides with another one at time $t$.
 These encounters happen in such a way that not more than two characteristics collide at the same time and space; see \cite[Proposition~3.2]{hora07b}.
 Moreover, since the ordering of the characteristics is preserved by the flow (see for example \cite[Theorem~3.4]{hora07b}), % since $y_\xi \geq 0$ by \cite[(2.12b)]{hora07b}
 % it also follows upon considering the second derivative of the characteristics, using \cite[(4.1)]{hora07b}
 that is, 
 \begin{align}\label{eqnCharOrd}
  - \infty < q_1(t) \leq \cdots \leq q_N(t) < \infty, \quad t\in\R,
 \end{align}
 only adjacent characteristics may coincide. 
 
 \begin{lemma}\label{lem:w_n}
 Assume that all characteristics are distinct at time $t\in\R$. Then 
 \begin{align}\label{eq:w=ju}
 \begin{pmatrix} 
  p_1(t) \\ p_2(t) \\ p_3(t) \\ \vdots \\ p_N(t) 
 \end{pmatrix}
  = \begin{pmatrix}
 a_1(t) & b_1(t) & 0 & \dots & 0\\
 b_{1}(t) & a_{2}(t) & b_{2}(t) & \dots & 0\\
 0 & b_{2}(t) & a_{3}(t) &  \dots & 0\\
 \vdots & \vdots &  \vdots & \ddots & \vdots\\
 0 & 0 & 0  & \dots & a_{N}(t)
 \end{pmatrix} 
 \begin{pmatrix} 
  u(q_1(t),t) \\ u(q_2(t),t) \\ u(q_3(t),t) \\ \vdots \\ u(q_N(t),t) 
 \end{pmatrix},
 \end{align}
 where 
 \begin{align*}
  a_{n}(t) & = \frac{1}{2} \frac{\sinh\left(q_{n+1}(t)-q_{n-1}(t)\right)}{\sinh\left(q_{n+1}(t)-q_n(t)\right)\sinh\left(q_{n}(t)-q_{n-1}(t)\right)}, & n & =1,\dots, N,\\ %\label{eq:enn+1}\\
  b_{n}(t) & = \frac{1}{2} \frac{-1}{\sinh\left(q_{n+1}(t)-q_n(t)\right)}, & n & =1,\dots,N-1. %\label{eq:enn}
 \end{align*}
 Hereby, we set $q_0(t)=-\infty$ and $q_{N+1}(t)=+\infty$ for notational simplicity. 
 \end{lemma}
 
 \begin{proof}
 It clearly follows from  \eqref{eqnuGCMP} that
  \begin{align*}
 \left(u(q_i(t),t)\right)_{i=1}^N = E(t) \left(p_i(t)\right)_{i=1}^N,\quad\text{with}\quad  
 E(t)=\left(\E^{-|q_j(t)-q_i(t)|}\right)_{i,j=1}^N.
 \end{align*}
 To complete the proof, it suffices to notice that the inverse of $E(t)$ is the tri-diagonal matrix in the claim (this can also be checked by a direct computation).
 \end{proof}
 
 We will denote the (closed) set of all times $t\in\R$ for which some characteristics coincide with $\Gamma$.  
 Upon considering the second derivative of the functions $q_n$ at a time $t\in\Gamma$ (employing \cite[(4.1)]{hora07b}), one infers that $\Gamma$ is a discrete set, consisting of isolated points.   
 % In fact, if $q_n(t^\times)= q_{n+1}(t^\times)$, then (in the notation of \cite{hora07b}; see \cite[(2.12b)]{hora07b})
 % \begin{align}
 %  q_{n+1}''(t^\times) - q_n''(t^\times) = \frac{1}{2} \left( H(q_{n+1}(t_0),t^\times) - H(q_{n}(t_0),t^\times) \right) > 0 
 % \end{align}
 As a consequence of Lemma~\ref{lem:w_n}, one sees that the heights $p_n$ are continuous away from $\Gamma$ (cf.\ \cite[(4.1)]{hora07b}). 
 In fact, the heights $p_n$ and positions $q_n$ of the single peaks, defined in~\eqref{eqnuGCMP}, even satisfy the system~\eqref{eqnMPsys} on $\R\backslash\Gamma$; \cite[Lemma~2.2]{hora06}.  
 This means that for times away from $\Gamma$, the solution $u$ is actually a classical multi-peakon solution. 
 In this respect, let us also recall that the measures $\mu(\,\cdot\,,t)$ are absolutely continuous (and hence uniquely determined by $u(\,\cdot\,,t)$) as long as $t\not\in\Gamma$. %; \cite[Proposition~3.2]{hora07b}. 
 More precisely, it has been proven in \cite[Proposition~3.2]{hora07b} that $\mu(\,\cdot\,,t)$ admits a singular part only when collisions occur.
 In this case, the singular part $\mu_{\text{s}}(\,\cdot\,,t)$ is supported on all points where peaks collide.
  
 Since classical multi-peakon solutions are fairly well understood, we are left to describe the global conservative solution $(u,\mu)$ at some time $t^\times\in\Gamma$ of collision. 
 We will do this in terms of limits of the quantities $p_n(t)$ and $q_n(t)$ (which are at least well-defined in a small vicinity of $t^\times$) as $t\rightarrow t^\times$. 
 First of all, if the characteristic $q_n(t^\times)$ is distinct from all the other ones for some $n\in\lbrace 1,\ldots,N\rbrace$, then $p_n(t^\times)$ is still uniquely determined by $u(\,\cdot\,,t^\times)$. 
 Moreover, it turns out that $p_n$ is continuous on the set $\lbrace t\in\R \,|\, q_{n-1}(t) < q_{n}(t) < q_{n+1}(t) \rbrace$ indeed and hence
 \begin{align}
  p_n(t^\times) = \lim_{t\rightarrow t^\times} p_n(t). 
 \end{align} 
 Of course, the more interesting case is when two characteristics coincide. 
 Therefore, suppose that we have $q_n(t^\times) = q_{n+1}(t^\times)$ for some $n=1,\ldots,N-1$.
 In this case, the corresponding heights $p_n(t)$ and $p_{n+1}(t)$ will blow up as $t\rightarrow t^\times$ % see below \eqref{eqnCo2} 
 and the two heights $p_n(t^\times)$ and $p_{n+1}(t^\times)$ are not well-defined by~\eqref{eqnuGCMP} anymore. 
 However, their sum $p_n(t^\times)+p_{n+1}(t^\times)$ is uniquely determined by $u(\,\cdot\,,t)$ and can readily be written down in a similar way as in Lemma~\ref{lem:w_n}. 
 Moreover, some amount of the energy of the solution will concentrate in the point of collision, which causes $\mu(\,\cdot\,,t^\times)$ to have positive mass at this point. 
 The following lemma describes all these quantities in terms of limits of the heights $p_n(t)$ and positions $q_n(t)$ of the single peaks as $t\rightarrow t^\times$. 
 
 \begin{lemma}\label{lem:2.3}
 Assume that $q_n(t^\times) = q_{n+1}(t^\times)$ for some $n=1,\ldots,N-1$. Then 
 \begin{align}\label{eqnCo1}
   p_n(t^\times) + p_{n+1}(t^\times) = \lim_{t\rightarrow t^\times} p_n(t)+ p_{n+1}(t) 
 \end{align}
 and the mass of $\mu(\,\cdot\,,t^\times)$ at the point of collision is given by  
 \begin{align}\label{eqnCo2}
  \mu(\lbrace q_n(t^\times) \rbrace, t^\times) = \lim_{t\rightarrow t^\times} 4\, p_n(t) p_{n+1}(t) (q_{n}(t) - q_{n+1}(t)).
 \end{align} 
 \end{lemma}
 
 \begin{proof}
  Using \eqref{eqnuGCMP}, for all $t\not=t^\times$ in a small vicinity of $t^\times$ we have 
 \begin{align}
\begin{split}\label{eq:p+p}
  2 (p_{n+1}(t) + p_n(t)) & =  u_x(q_{n+1}(t)-,t)  - u_x(q_n(t)+,t) \\ 
                          & \qquad\qquad\qquad\qquad + u_x(q_n(t)-,t) - u_x(q_{n+1}(t)+,t).
 \end{split}
 \end{align} 
 The sum of the two terms in the upper line is seen to converge to zero as $t\rightarrow t^\times$ upon noting that   
 \begin{align*} 
  u_x(q_n(t)+,t) & =  \frac{u(q_n(t),t) (1+\E^{2(q_{n+1}(t)-q_n(t))}) - 2 u(q_{n+1}(t),t) \E^{q_{n+1}(t)-q_n(t)}}{1-\E^{2(q_{n+1}(t)-q_n(t))}}, \\
  u_x(q_{n+1}(t)-,t) & = \frac{2 u(q_{n}(t),t) \E^{q_{n+1}(t)-q_n(t)} - u(q_{n+1}(t),t) (1+\E^{2(q_{n+1}(t)-q_n(t))})}{1-\E^{2(q_{n+1}(t)-q_n(t))}}.
 \end{align*}
 Using similar representations for $u_x(q_n(t)-,t)$ and $u_x(q_{n+1}(t)+,t)$, we see that the two terms in the lower line of~\eqref{eq:p+p} are continuous near $t^\times$ and hence their sum converges to $2(p_{n+1}(t^\times) + p_n(t^\times))$ as $t\rightarrow t^\times$.  

 In order to prove the second claim, we use \cite[equation~(4.8)]{hora07b} and integrate by parts to obtain 
  \begin{align*}
  \mu(\lbrace q_n(t^\times) \rbrace, t^\times) & = \lim_{t\rightarrow t^\times} \int_{q_n(t)}^{q_{n+1}(t)} u(r,t)^2 + u_x(r,t)^2 dr \\
                                               & = \lim_{t\rightarrow t^\times} u_x(q_{n+1}(t)-,t) u(q_{n+1}(t),t) - u_x(q_n(t)+,t) u(q_n(t),t) \\
                                               & = \lim_{t\rightarrow t^\times} u_x(q_n(t)+,t) ( u(q_{n+1}(t),t) - u(q_n(t),t) ). 
  \end{align*}  
  Taking into account the relation (employing equation for $u_x(q_{n+1}(t)-,t)$ above)  
  \begin{align*}
   (q_{n+1}(t) - q_n(t)) u_x(q_{n+1}(t)-,t) \sim u(q_{n+1}(t),t) - u(q_n(t),t)   
  \end{align*}
  as $t\rightarrow t^\times$, one furthermore gets  
  \begin{align*}
  \mu(\lbrace q_n(t^\times) \rbrace, t^\times) % & = \lim_{t\rightarrow t^\times} u_x(q_n(t)+,t) (u(q_{n+1}(t),t) - u(q_n(t),t)) \\  
                                               % & = \lim_{t\rightarrow t^\times} u_x(q_n(t)+,t) (q_{n+1}(t)-q_n(t)) 2 (u(q_n(t),t) - u(q_{n+1}(t),t)) \frac{\E^{q_{n+1}(t)-q_n(t)}}{1-\E^{2(q_{n+1}(t)-q_n(t))}} \\  
                                                 & = \lim_{t\rightarrow t^\times} (q_{n+1}(t)-q_n(t)) u_x(q_n(t)+,t) u_x(q_{n+1}(t)-,t) \\
                                                 & = \lim_{t\rightarrow t^\times} 4\, p_n(t) p_{n+1}(t) (q_n(t) - q_{n+1}(t)). 
  \end{align*}
 % For the last step note that $2p_n(t) = u_x(q_n(t)-,t) - u_x(q_n(t)+,t)$ and hence $2 p_n(t) \sim - u_x(q_n(t)+,t)$ since $u_x(q_n(t)+,t)$ blows up and $u_x(q_n(t)-,t)$ stays bounded. 
 % Similarly, one gets $2 p_{n+1}(t) \sim u_x(q_{n+1}(t)-,t)$.
\end{proof}

  \section{The generalized spectral problem}\label{s1}

In order to introduce our generalized spectral problem, fix some $N\in\N_0$ and let $x_1,\ldots,x_N\in\R$ be strictly increasing. 
 Moreover, for each $n\in\lbrace 1,\ldots,N\rbrace$, let $\omega_n\in\R$, $\dip_n\geq 0$ and consider the finite discrete measures 
 \begin{align}\label{eqnMEA}
 \omega = \sum_{n=1}^N \omega_n \delta_{x_n} \quad\text{and}\quad \dip = \sum_{n=1}^N \dip_n \delta_{x_n},
\end{align}
 where $\delta_{x_n}$ is the Dirac measure at $x_n$. 
 For definiteness, we will also assume that $\supp(|\omega|+\dip) = \lbrace x_1,\ldots,x_N\rbrace$, that is, $|\omega_n|+\dip_n>0$ for $n=1,\ldots,N$.  
 In this section we will consider the spectral problem
\begin{align}\label{eqnDE}
 -f''(x) + \frac{1}{4} f(x) = z\, \omega(x) f(x) + z^2 \dip(x) f(x), \quad x\in\R,
\end{align}
with a complex spectral parameter $z\in\C$. 
 Of course, since $\omega$ and $\dip$ are measures, this equation has to be understood in a distributional sense (the right-hand side is a measure if $f$ is continuous).
 %\footnote{The spectral problem \eqref{eqnDE} can also be rewritten in the integral form\\
 % the integral equation
 %\begin{align}\label{eqnIE}
 % $f(x)= z \int_\R\E^{-\frac{|x-s|}{2}}f(s)d\omega(s)  + z^2 \int_\R\E^{-\frac{|x-s|}{2}}f(s)d\dip(s)$.
 %\end{align}
 % }
 More precisely, some function $f$ is a solution of the differential equation~\eqref{eqnDE} if it satisfies 
 \begin{align}\label{eqnDEdiscr01}
   -f''(x) + \frac{1}{4} f(x) = 0, \quad x\in\R\backslash\lbrace x_1,\ldots,x_N\rbrace, 
 \end{align}
 together with the interface conditions  
 \begin{align}
  \begin{pmatrix} f(x_n-) \\ f'(x_n-) \end{pmatrix}  =
\begin{pmatrix} 1 & 0  \\  z \omega_n + z^2 \dip_n  & 1 \end{pmatrix}
\begin{pmatrix} f(x_n+) \\ f'(x_n+) \end{pmatrix},\quad n\in\{1,\dots, N\}.
\label{eqnDEdiscr02}
 \end{align}
 Hereby note that in this case, the solution $f$ is in general not differentiable at the points $x_1,\ldots,x_N$. 
 However, for simplicity of notation, we will always uniquely extend the derivative $f'$ to all of $\R$ by requiring it to be left-continuous. 

 The set of all values $z\in\C$ for which there is a nontrivial bounded solution of the differential equation~\eqref{eqnDE} is referred to as the spectrum $\sigma$ of the spectral problem~\eqref{eqnDE}. 
 Note that in this case, the bounded solution of this differential equation is unique up to scalar multiples. 

Since the measures $\omega$ and $\dip$ have compact support, for each $z\in\C$ one has spatially decaying solutions $\phi_\pm(z,\cdot\,)$ of~\eqref{eqnDE} with 
\begin{align}\label{eqnPHIAsym}
 \phi_\pm(z,x) = \E^{\mp\frac{x}{2}} 
\end{align}
 for all $x$ near $\pm\infty$. 
 In particular, note that $\phi_\pm(\,\cdot\,,x)$ and $\phi_\pm'(\,\cdot\,,x)$ are real polynomials for each fixed $x\in\R$. 
 The Wronski determinant of these solutions  
\begin{align}\label{eqnWronski}
 W(z) = \phi_+(z,x) \phi_-'(z,x) - \phi_+'(z,x) \phi_-(z,x), \quad z\in\C,
\end{align}
 is independent of $x\in\R$ (see for example \cite[Proposition~3.2]{MeasureSL}) and vanishes in some point $\lambda\in\C$ if and only if the solutions $\phi_-(\lambda,\cdot\,)$ and $\phi_+(\lambda,\cdot\,)$ are linearly dependent. 
 In this case, there is a nonzero constant $c_{\lambda}\in\C$ such that
 \begin{align}\label{eqnCC}
  \phi_-(\lambda,x) = c_{\lambda} \phi_+(\lambda,x), \quad x\in\R.   % This c_\lambda is c_{\lambda,-} in IsospecCH.
 \end{align}
 As a consequence, one sees that the spectrum $\sigma$ is precisely the set of zeros of the polynomial $W$. 
 Associated with each eigenvalue $\lambda\in\sigma$ is the quantity
  \begin{align}\label{eqnNC}
  \gamma_{\lambda}^2 & = \int_\R |\phi_+(\lambda,x)|^2 d\omega(x) + 2\lambda \int_\R |\phi_+(\lambda,x)|^2 d\dip(x),  \end{align}
 which is referred to as the (modified) norming constant (associated with $\lambda$).

 \begin{proposition}\label{propSR}
  Each eigenvalue $\lambda$ of the spectral problem~\eqref{eqnDE} is real with 
    \begin{align}\label{eqnWdot}
    - \dot{W}(\lambda) %= -\int_\R \phi_-(\lambda,x) \phi_+(\lambda,x) d\omega(x) - 2\lambda \int_\R \phi_-(\lambda,x) \phi_+(\lambda,x) d\dip(x) 
   = c_{\lambda} \gamma_\lambda^2 \not=0,
  \end{align}
  where the dot denotes differentiation with respect to the spectral parameter. 
 \end{proposition}
 
 \begin{proof}
  Let $\lambda$ be an eigenvalue of the spectral problem~\eqref{eqnDE}. 
   An integration by parts, using the differential equation~\eqref{eqnDE} shows
   \begin{align}
  \begin{split}\label{eqnPI}
  \lambda \int_\R |\phi_+(\lambda,x)|^2 d\omega(x) + & \lambda^2 \int_\R |\phi_+(\lambda,x)|^2 d\dip(x) \\ 
   & \quad = \frac{1}{4} \int_\R |\phi_+(\lambda,x)|^2 dx + \int_\R |\phi_+'(\lambda,x)|^2 dx > 0.
  \end{split}
 \end{align}
  If $\lambda$ was non-real, then an inspection of the imaginary part of the left-hand side would yield
  \begin{align*}
  \int_\R |\phi_+(\lambda,x)|^2 d\omega(x) + 2\re(\lambda) \int_\R  |\phi_+(\lambda,x)|^2 d\dip(x) = 0.
  \end{align*}
  Furthermore, computing the real part one would end up with the contradiction 
  \begin{align*}
   - |\lambda|^2 \int_\R |\phi_+(\lambda,x)|^2 d\dip(x) > 0,
  \end{align*}
  since the measure $\dip$ is non-negative.   

  In order to prove the second claim, we introduce the auxiliary function
  \begin{align*}
   W_0(z,x) = \phi_+(z,x) \dot{\phi}_-'(z,x) - \phi_+'(z,x) \dot{\phi}_(z,x), \quad x\in\R,~z\in\C,
  \end{align*}
  where the spatial differentiation is done first. 
  Using the differential equation which $\phi_\pm(\lambda,\cdot\,)$ satisfies, one gets (in a distributional sense)  
  \begin{align*}
   - W_0'(\lambda,x) = \phi_-(\lambda,x) \phi_+(\lambda,x) \omega(x) + 2\lambda \phi_-(\lambda,x) \phi_+(\lambda,x) \dip(x), \quad x\in\R.
  \end{align*}
  Now the claim follows upon noting that $W_0(\lambda,x)=0$ for every $x$ near $-\infty$ and $W_0(\lambda,x)=\dot{W}(\lambda)$ for $x$ near $\infty$.   
  In order to finish the proof, note that 
  \begin{align*}
  \lambda \gamma_\lambda^2 = \lambda \int_\R |\phi_+(\lambda,x)|^2 d\omega(x) + 2\lambda^2 \int_\R |\phi_+(\lambda,x)|^2 d\dip(x)
  \end{align*}
  is greater or equal than the left-hand side of~\eqref{eqnPI} and hence non-zero.  
 \end{proof}
 
 As a consequence, all zeros of $W$ are simple and hence the polynomial $W$ has the product representation (upon noting that $W(0)=1$) 
 \begin{align}\label{eqnW}
  W(z) = \prod_{\lambda\in\sigma} \biggl( 1-\frac{z}{\lambda}\biggr), \quad z\in\C. 
 \end{align}
 For applications to the Camassa--Holm equation, trace formulas for the spectral problem~\eqref{eqnDE} are of particular interest. 
 In fact, they will reappear as conserved quantities for global conservative multi-peakon solutions in Section~\ref{secGCMPS}. 

 \begin{proposition}\label{proptraceformulas}
  The first two trace formulas for the spectral problem~\eqref{eqnDE} are given by: 
  \begin{align}\label{eqnTF}
   \sum_{\lambda\in\sigma} \frac{1}{\lambda} & = \int_\R d\omega, &  \sum_{\lambda\in\sigma} \frac{1}{\lambda^2} & = \int_\R \int_\R \E^{-|x-s|} d\omega(s)\, d\omega(x) + 2 \int_\R d\dip.
 \end{align}
 \end{proposition}
 
 \begin{proof}
  Upon viewing~\eqref{eqnDE} as a perturbed equation (see \cite[Proposition~~3.3]{MeasureSL} or the proof of \cite[Theorem~3.1]{IsospecCH}), one sees that for each $z\in\C$ the solution $\phi_\pm(z,\cdot\,)$ satisfies the integral equation
%   \begin{align*}
%  \begin{split}
%   \phi_\pm(z,x) = \E^{\mp\frac{x}{2}} & \pm z \int_x^{\pm\infty} \left(\E^{\pm\frac{x-s}{2}} - \E^{\mp\frac{x-s}{2}}\right) \phi_\pm(z,s) d\omega(s) \\
%   & \pm z^2 \int_x^{\pm\infty} \left(\E^{\pm\frac{x-s}{2}} - \E^{\mp\frac{x-s}{2}}\right) \phi_\pm(z,s) d\dip(s), \quad x\in\R,
%   \end{split}
%  \end{align*}
  \begin{align}\label{eqnIEpm}
  \begin{split}
   \phi_\pm(z,x) = \E^{\mp\frac{x}{2}} & + z \int_{I_\pm(x)} \left(\E^{\pm\frac{x-s}{2}} - \E^{\mp\frac{x-s}{2}}\right) \phi_\pm(z,s) d\omega(s) \\
   & + z^2 \int_{I_\pm(x)} \left(\E^{\pm\frac{x-s}{2}} - \E^{\mp\frac{x-s}{2}}\right) \phi_\pm(z,s) d\dip(s), \quad x\in\R,
   \end{split}
  \end{align}
  where $I_-(x) = (-\infty,x)$ and $I_+(x) = [x,\infty)$.
  In particular, this yields   
%  \begin{align*}
%   \dot{\phi}_\pm(0,x)  & = \pm \E^{\mp\frac{x}{2}} \int_x^{\pm\infty} \left(\E^{-|x-s|} - 1\right) d\omega(s), \\
%   \dot{\phi}_\pm'(0,x) & = \frac{1}{2} \E^{\mp\frac{x}{2}} \int_x^{\pm\infty} \left(\E^{-|x-s|} + 1\right) d\omega(s).
%  \end{align*}
  \begin{align*}
   \dot{\phi}_\pm(0,x)  & =  \E^{\mp\frac{x}{2}} \int_{I_\pm(x)} \left(\E^{-|x-s|} - 1\right) d\omega(s),  \\
   \dot{\phi}_\pm'(0,x) & = \pm \frac{1}{2} \E^{\mp\frac{x}{2}} \int_{I_\pm(x)} \left(\E^{-|x-s|} + 1\right) d\omega(s),
  \end{align*}
%  as well as
%  \begin{align*}
%   \ddot{\phi}_\pm(0,x)  & = 2 \E^{\mp\frac{x}{2}} \int_x^{\pm\infty} \left(\E^{-|x-s|} - 1\right) \int_s^{\pm\infty} \left(\E^{-|s-r|} - 1\right) d\omega(r)\, d\omega(s) \\
%                         & \hspace{5.5cm} \pm 2 \E^{\mp\frac{x}{2}}  \int_x^{\pm\infty} \left(\E^{-|x-s|} - 1\right) d\dip(s), \\
%   \ddot{\phi}_\pm'(0,x) & =  \pm \E^{\mp\frac{x}{2}} \int_x^{\pm\infty} \left(\E^{-|x-s|} + 1\right) \int_s^{\pm\infty} \left(\E^{-|s-r|} 1 1\right) d\omega(r)\, d\omega(s) \\
%                         & \hspace{5.5cm} +  \E^{\mp\frac{x}{2}}  \int_x^{\pm\infty} \left(\E^{-|x-s|} + 1\right) d\dip(s).
%  \end{align*} 
%%  where we used the abbreviation
%%  \begin{align*}
%%   K(x,s) = \E^{-|x-s|} - 1 \quad\text{and}\quad K'(x,s) = \E^{-|x-s|} +1, \quad x,\, s\in\R.
%%  \end{align*}
  for each $x\in\R$.
  Using these equalities (and noting that $\phi_+(z,x)$ and $\phi_+'(z,x)$ do not depend on $z$ as long as $x$ is chosen near $+\infty$), one obtains  
  \begin{align*}%\label{eqnWexp}
   - \dot{W}(0) = -\left( \E^{-\frac{x}{2}} \dot{\phi}_-'(0,x) + \frac{1}{2} \E^{-\frac{x}{2}} \dot{\phi}_-(0,x)\right) = \int_\R d\omega % \quad\text{and}\quad \ddot{W}(0) = -2 \int_\R d\dip - \int_\R \int_\R \E^{-|s-r|}d\omega(r)\, d\omega(s).
  \end{align*}
  for $x$ near $+\infty$.  
  After a similar, but more lengthy calculation, one ends up with the second derivative of $W$ at zero given by 
  \begin{align*}
   - \ddot{W}(0) =  \int_\R \int_\R \left(\E^{-|x-s|} - 1\right) d\omega(s)\, d\omega(x) +  2\int_\R d\dip.
  \end{align*}  
%  More precisely, since $\omega$ and $\dip$ have compact support, we have
%  \begin{align*}
%   W(z) = \pm\E^{\mp\frac{x}{2}} \phi_\mp'(z,x) + \frac{1}{2}\E^{\mp\frac{x}{2}} \phi_\mp(z,x), \quad z\in\C
%  \end{align*}
%  for all $x$ near $\pm\infty$. 
%  Now plugging in the above equations and letting $x\rightarrow\pm\infty$ one ends up with 
%  \begin{align*}
%   \ddot{W}(0) = \pm 2 \int_\R \int_s^{\mp\infty} \left(\E^{-|s-r|} - 1 \right) d\omega(r)\, d\omega(s) - 2\int_\R d\dip.
%  \end{align*}
%  Upon noting that these are actually two equations for $\ddot{W}(0)$ and adding them, one obtains the second equation in~\eqref{eqnWexp}.  
  Now the trace formulas in the claim follow immediately from the product representation~\eqref{eqnW} of the Wronskian $W$.
  \end{proof}

% Using the identities in the proof of Proposition~\ref{proptraceformulas}, one also obtains the following useful expansion of the Green's function near zero. 

 % \begin{proposition}\label{propGexp}
%  For each $x\in\R$ we have the asymptotics
%  \begin{align}
%   \frac{\phi_-(z,x) \phi_+(z,x)}{W(z)} = 1 + z \int_\R \E^{-|x-s|} d\omega(s) + \OO(z^2),
%  \end{align}
%  as $z\rightarrow 0$ in $\C$. 
% \end{proposition}

 \section{The inverse spectral problem}\label{secIP}

 As already mentioned in the introduction, the corresponding inverse problem has been solved (in the form of the equivalent Krein--Stieltjes string) by Krein and Langer in \cite{krla79, krla80}, employing appropriate generalizations of a theorem of Stieltjes on continued fractions. 
 Nevertheless, we will give a proof of this result in our terminology for the sake of completeness and the convenience of the reader. 
 Therefore, we first introduce the rational Weyl--Titchmarsh function $M$ on $\C\backslash\R$, associated with the spectral problem~\eqref{eqnDE}, via 
 \begin{align}\label{eqnM}
  M(z)W(z) = \frac{1}{2} \E^{\frac{x}{2}} \phi_-(z,x)- \E^{\frac{x}{2}} \phi_-'(z,x), \quad z\in\C\backslash\R,
 \end{align}
 for $x$ near $+\infty$ (note that the Wronskian on the right-hand side is constant there).

 \begin{lemma}\label{lem:herg}
 The Weyl--Titchmarsh function $M$ admits the following partial fraction expansion:  
 \begin{align}\label{eqn:3.2}
 \frac{M(z)}{z} =  \sum_{\lambda\in\sigma} \frac{\gamma_\lambda^{-2}}{\lambda(\lambda-z)}, \quad z\in\C\backslash\R.
                  %  = \int_\R \frac{1}{\lambda(\lambda-z)} d\rho(\lambda)
 \end{align}
 In particular, the function in \eqref{eqn:3.2} is a Herglotz--Nevanlinna function.
 \end{lemma}
 
 \begin{proof}
 By definition, the poles of $M$ are simple and precisely the eigenvalues of the spectral problem~\eqref{eqnDE}. 
 Moreover, the residues of $M$ can be obtained immediately from Proposition~\ref{propSR} (also note that $M(0)=0$). 
 Indeed, for each $\lambda\in \sigma$ we get   
  \[
  \lim_{z\to \lambda}\frac{M(z)}{z}(\lambda-z)= - \frac{c_\lambda}{\lambda\dot{W}(\lambda)} = \frac{1}{\lambda\gamma_\lambda^2},
  \]
  upon employing \eqref{eqnM}, \eqref{eqnWdot} and \eqref{eqnCC}. 
  Therefore, the expansion \eqref{eqn:3.2} immediately follows if $M(z)=\OO(1)$ as $|z|\to \infty$. 
  In order to prove this, first one notes that the integral equation~\eqref{eqnIEpm} for our solution implies for each $n\in\lbrace 1,\ldots,N-1\rbrace$ 
  \begin{align*}
   \phi_-(z,x_{n+1}) \asymp \phi_-(z,x_n) (z\omega_n + z^2 \dip_n), \quad |z|\rightarrow \infty.
  \end{align*}
  Moreover, upon differentiating~\eqref{eqnIEpm} one also infers that for $n\in\lbrace 1,\ldots,N\rbrace$ 
  \begin{align*}
   \phi_-'(z,x_n) \asymp \phi_-(z,x_n), \quad |z|\rightarrow \infty.
  \end{align*} 
  As a consequence, a simple calculation shows that for every $x>x_N$ 
  \begin{align*}
   \E^{\pm \frac{x}{2}} \phi_-'(z,x) \mp \frac{1}{2} \E^{\pm\frac{x}{2}} \phi_-(z,x) \sim - \E^{\pm\frac{x_N}{2}} \phi_-(x_N,z) (z\omega_N+z^2\dip_N), \quad |z|\rightarrow \infty. 
  \end{align*}
  Hereby, we used the interface condition~\eqref{eqnDEdiscr02} and the fact that the left-hand side is constant to the right of $x_N$.  
  In view of the definition of the Weyl--Titchmarsh function, this shows that $M(z)\rightarrow - \E^{x_N}$ as $|z|\to \infty$, proving~\eqref{eqn:3.2}. 
  Finally, the function in~\eqref{eqn:3.2} is a Herglotz--Nevanlinna function indeed, since its residues are negative in view of~\eqref{eqnPI}. 
 \end{proof}
 
 On the other side, it is possible to write down a finite continued fraction expansion for $M$ in terms of $\omega$ and $\dip$. 
 Therefore, we introduce the quantities
 \begin{align}\label{eqnan}
  l_n & =  \frac{1}{2} \left( \tanh\left(\frac{x_{n+1}}{2}\right) - \tanh\left(\frac{x_n}{2}\right)\right), \quad n= 0,\ldots,N,
  \end{align}
  and the polynomials 
 \begin{align}\label{eqnbn}
  m_n(z) & = \left(z\omega_n + z^2\dip_n\right) 4 \cosh^2\left(\frac{x_n}{2}\right), \quad z\in\C,~ n=1,\ldots,N.
 \end{align}
 Hereby, we set $x_0=-\infty$ and $x_{N+1}=+\infty$ for simplicity of notation.
 
 \begin{lemma}\label{lemCF}
  The Weyl--Titchmarsh function $M$ admits the following finite continued fraction expansion:  
  \begin{align}\label{eqnContFrac}
   M(z) = 1 + \cfrac{1}{-l_{N} + \cfrac{1}{m_N(z) + \cfrac{1}{\;\ddots\; + \cfrac{1}{-l_1 + \cfrac{1}{ m_1(z) - \cfrac{1}{l_0}}}}}} \qquad z\in\C\backslash\R.
  \end{align}
 \end{lemma}

 \begin{proof}
  For each fixed $z\in\C\backslash\R$, consider the function $\Lambda(z,\cdot\,)$ on $\R$ given by 
  \begin{align}\label{eqnM0}
\Lambda(z,x) = 2\sinh\left(\frac{x}{2}\right) \cosh\left(\frac{x}{2}\right) - 4 \frac{\phi_-'(z,x)}{\phi_-(z,x)} \cosh^2\left(\frac{x}{2}\right), \quad x\in\R.
  \end{align}
  Since $\omega$ and $\dip$ have compact support (and $z$ is non-real), we have the asymptotics 
  \begin{align*}
   2\phi_-'(z,x) = \phi_-(z,x) \left(1+\OO\left(\E^{-x}\right)\right)
  \end{align*}
  as $x\rightarrow \infty$ and as a consequence (after some calculations)
  \begin{align*}
   M(z) = \lim_{x\rightarrow\infty} \Lambda(z,x) + 1.
  \end{align*}  
  Now, first of all note that because of~\eqref{eqnPHIAsym} we have 
  \begin{align*}
   \Lambda(z,x_1-) = -\left(1+\E^{-x_1}\right) = - \frac{1}{l_0}.
  \end{align*}
  Moreover, because of the interface condition \eqref{eqnDEdiscr02}, we end up with
  \begin{align*}
  \Lambda\left(z,x_n+\right)-\Lambda\left(z,x_n-\right)=(z\omega_n+z^2\dip_n)4\cosh^2\left(\frac{x_n}{2}\right), \quad n=1,\ldots,N.
  \end{align*}
  Upon differentiating~\eqref{eqnM0}, one furthermore obtains 
  \begin{align*}
   \Lambda'(z,x) = \frac{\Lambda(z,x)^2}{4\cosh^2\left(\frac{x}{2}\right)}, 
  \end{align*}
  for all real $x$ away from $\lbrace x_1,\ldots,x_N\rbrace$, which immediately implies 
  \begin{align*}
  \frac{1}{\Lambda(z,x_{n+1}-)} - \frac{1}{\Lambda(z,x_{n}+)}= - l_{n}, \quad n=1,\ldots,N. 
  \end{align*}
  Thus, we finally arrive at the following relation 
  \begin{align*}
    \frac{1}{\Lambda\left(z,x_{n+1}-\right)} = - l_{n} + \frac{1}{m_{n}(z) + \Lambda\left(z,x_{n}-\right)}, \quad n=1,\ldots,N.
  \end{align*} 
  Hereby, for $n=N$, the denominator on the left-hand side has to be read as
  \begin{align*}
   \Lambda\left(z,x_{N+1}-\right) = \lim_{x\rightarrow\infty} \Lambda\left(z,x\right),
   \end{align*}
  which completes the proof.  
 \end{proof} 
   
 We are now able to solve the inverse spectral problem following \cite[\S 3.4]{krla79}.  
 
 \begin{theorem}\label{thmIP}
  Let $\sigma$ be a finite subset of $\R$ and for each $\lambda\in\sigma$ let $\gamma_\lambda^2\in\R$ be such that $\lambda\gamma_\lambda^{2}>0$.
  Then there are unique measures $\omega$ and $\dip$ of the form in~\eqref{eqnMEA} such that the corresponding spectrum is $\sigma$ and the norming constants are $\gamma_\lambda^{2}$ for $\lambda\in\sigma$. 
 \end{theorem}

 \begin{proof}
 {\em Existence.} Let $\sigma$ be a finite subset of $\R$ and for each $\lambda\in\sigma$ let $\gamma_\lambda^2\in\R$ such that $\lambda\gamma_\lambda^{2}>0$.
 First of all, we will show that the rational function $M$ defined by  
 \begin{align*}
  M(z) = z \sum_{\lambda\in\sigma} \frac{\gamma_\lambda^{-2}}{\lambda(\lambda-z)}, \quad z\in\C\backslash\R,  
 \end{align*}
 has a finite continued fraction expansion of the form~\eqref{eqnContFrac}.  
 Therefore, we introduce for each $k\in\N$ the following determinants of Hankel matrices 
 \begin{align*} 
  \Delta_{0,k} & = \begin{vmatrix} s_0 & s_1 & \cdots & s_{k-1} \\ s_1 & s_2 & \cdots & s_{k} \\ \vdots & \vdots & \ddots & \vdots \\ s_{k-1} & s_{k} & \cdots & s_{2k-2} \end{vmatrix}, &  
  \Delta_{1,k} & = \begin{vmatrix} s_1 & s_2 & \cdots & s_{k} \\ s_2 & s_3 & \cdots & s_{k+1} \\ \vdots & \vdots & \ddots & \vdots \\ s_{k} & s_{k+1} & \cdots & s_{2k-1} \end{vmatrix},
  \end{align*}
  as well as 
  \begin{align*}
  \Delta_{-1,k} & = \begin{vmatrix} s_{-1} & s_0 & \cdots & s_{k-2} \\ s_0 & s_1 & \cdots & s_{k-1} \\ \vdots & \vdots & \ddots & \vdots \\ s_{k-2} & s_{k-1} & \cdots & s_{2k-3} \end{vmatrix}, & 
  \Delta_{2,k} & = \begin{vmatrix} s_2 & s_3 & \cdots & s_{k+1} \\ s_3 & s_4 & \cdots & s_{k+2} \\ \vdots & \vdots & \ddots & \vdots \\ s_{k+1} & s_{k+2} & \cdots & s_{2k} \end{vmatrix}.
 \end{align*}
We also set $\Delta_{0,0}=\Delta_{1,0}=\Delta_{2,0}=1$ for notational simplicity.
 Hereby, the quantities $s_k$ are determined by the asymptotic expansion of the function 
 \begin{align}\label{eqnMHN}
  \frac{1-M(z)}{z} = \sum_{k=0}^{\infty} \frac{s_{k-1}}{z^{k}}, \quad |z|\rightarrow \infty. % \sum_{k=0}^{2|\sigma| + 1} \frac{s_k}{z^{k+1}} % \frac{s_0}{z}+\frac{s_1}{z^2}+\frac{s_2}{z^3}+\dots+\frac{s_{2|\sigma|+1}}{z^{2|\sigma|+2}}
                                                               %    + \OO\left(\frac{1}{z^{K+1}}\right),
  \end{align}
  More precisely, it follows from the definition of the function $M$ that they are given explicitly in terms of the data by  
  \begin{align*}
  s_{-1} = 0, \quad s_0=1 + \sum_{\lambda\in\sigma} \frac{1}{\lambda\gamma_\lambda^{2}} \quad\text{and}\quad  
  s_k=  \sum_{\lambda\in\sigma}\frac{\lambda^{k}}{\lambda\gamma_\lambda^{2}},\quad k\in\N. 
 \end{align*}
 Next, one notes that the determinants $\Delta_{0,k}$ for $k= 1,\ldots,|\sigma|+1$ and $\Delta_{2,k}$ for $k=1,\ldots,|\sigma|$ are positive (see \cite{akh}, \cite[Theorem~3.4]{hoty12} and observe that the function in \eqref{eqnMHN} as well as the function
 \begin{align}\label{eqnM2HN}
  - \sum_{\lambda\in\sigma} \frac{\lambda \gamma_\lambda^{-2}}{\lambda-z} = \sum_{k=1}^\infty \frac{s_{k+1}}{z^k}, \quad |z|\rightarrow \infty,
 \end{align}
 are anti-Herglotz--Nevanlinna functions), with $|\sigma|$ denoting the number of elements of $\sigma$. 
 Furthermore, there are no consecutive zeros within the sequence $\Delta_{1,k}$, $k=1,\ldots,|\sigma|+1$ since for every $k\in\N$ one has the relation 
 \begin{align}\label{eqnDeltaRel}
  \Delta_{1,k+1}  \Delta_{-1,k+1}  - \Delta_{1,k} \Delta_{-1,k+2}= \Delta_{0,k+1}^2,
 \end{align} 
 which follows from Sylvester's determinant identity \cite{ba68}, \cite{ga}. 
 Namely, one needs to consider the determinant $ \Delta_{-1,k+2}$ and then apply formula \cite[(II.28)]{ga} (with $p=n-2$). 
 In particular, it follows from~\eqref{eqnDeltaRel} with $k=|\sigma|$ that $\Delta_{1,|\sigma|}\neq 0$ since $\Delta_{1,|\sigma|+1}=0$ (cf.\ \cite[Corollary~1.4]{hoty12}).

 In the case when the quantities $\Delta_{1,k}$ are non-zero for all $k=1,\ldots,|\sigma|$, a theorem of Stieltjes (see \cite[Theorem~1.39]{hoty12}) 
 % applied to
 % \begin{align*}
 % - \frac{M(z)-1}{z} = \frac{1}{z} \left( 1 + \sum_{\lambda\in\sigma} \frac{\gamma_\lambda^{-2}}{\lambda} \right) + \sum_{k=1}^\infty \frac{1}{z^{k+1}} \sum_{\lambda\in\sigma} \lambda^k \frac{\gamma_\lambda^{-2}}{\lambda}
 % \end{align*}
 guarantees that the function $M$ has a finite continued fraction expansion of the form~\eqref{eqnContFrac} with $N=|\sigma|$ and the coefficients given by  
  \begin{align}\label{eqnCoanInv}
  l_n = \frac{\Delta_{1,N-n}^2}{\Delta_{0,N-n} \Delta_{0,N-n+1}} > 0,  
 \end{align}
 for $n=0,\ldots,N$ and by 
 \begin{align}\label{eqnCobnInv}
  m_n(z) = z \frac{\Delta_{0,N-n+1}^2}{\Delta_{1,N-n} \Delta_{1,N-n+1}}, \quad z\in\C,
 \end{align}
 for $n=1,\ldots,N$. 
 % with the notation from \cite{hoty12} $D_k = \Delta_{0,k}$ and $\hat{D}_k = \Delta_{1,k}$
 % from \cite[Corollary~1.4]{hoty12} $\Delta_{1,|\sigma|} \not=0$ and $\Delta_{1,k} = 0$ for $k\geq |\sigma| + 1$. 
 % from \cite[Theorem~1.2 and Theorem~1.3]{hoty12} $\Delta_{0,k} \not=0 for 
 
 Otherwise (that is, if $\Delta_{1,k}$ is zero for some $k\in\lbrace1,\ldots,|\sigma|\rbrace$), we consider the rational functions $M_t$ given by\footnote{ %Of course, the exponent $\E^{t\lambda}$ is a more natural choice for $t$-dependence. 
 Our choice for this perturbation of $M$ is motivated by the time evolution of the spectral data under the Camassa--Holm flow (cf.\ Theorem~\ref{thmTE}).}
 \begin{align}\label{eq:M_t}
  M_t(z) = z \sum_{\lambda\in\sigma} \frac{\gamma_\lambda^{-2}}{\lambda(\lambda-z)} \E^{\frac{t}{2\lambda}}, \quad z\in\C\backslash\R,
 \end{align}
 for each $t\in\R$. 
 Of course, the above considerations apply to these functions as well and we denote the corresponding quantities with an additional subscript $t$. 
 In particular, note that $M_0$ coincides with our initial function $M$ and hence so do the associated quantities (thus in the case $t=0$ we will omit the additional subscripts). 
 Since  for each $k\in\N$ the determinants $\Delta_{t,0,k}$ and $\Delta_{t,1,k}$ depend analytically on $t$, we may conclude that the quantities $\Delta_{t,1,k}$, $k=1,\ldots,|\sigma|$ are non-zero for all small enough $t\not=0$. 
 In fact, even if $\Delta_{1,k}$ vanishes for some $k\in\lbrace1,\ldots,|\sigma|-1\rbrace$ (note that $\Delta_{t,1,|\sigma|}\neq 0$ for all $t\in\R$), then this holds since the derivative at zero 
 \begin{align*}
   \Delta_{1,k}' :=\left[\frac{d}{dt} \Delta_{t,1,k}\right]_{t=0} =  2^{-k}\begin{vmatrix} {s}_0-1 & s_2  & s_{3}& \cdots & s_{k} \\ s_1 & s_3  & s_{4} & \cdots & s_{k+1} \\ \vdots & \vdots & \vdots & \ddots & \vdots \\ s_{k-2} & s_{k}  & s_{k+1}& \cdots & s_{2k-2} \\ s_{k-1} & s_{k+1}  & s_{k+2} & \cdots & s_{2k-1} \end{vmatrix} 
 \end{align*}
 is non-zero (in order to compute this derivative, just observe that 
 \begin{align*}
  \frac{d}{dt} s_{t,1} = \frac{1}{2}\left(s_{t,0}-1\right) \quad\text{and}\quad  \frac{d}{dt} s_{t,k+1} = \frac{1}{2} s_{t,k}, \quad  k\in\N, 
 \end{align*}
 by definition of $M_t$). 
 Indeed, non-vanishing of the above derivative follows from the relation  
 \begin{align}\label{eq:deriv}
 2^k\Delta_{1,k}' \cdot
 \begin{vmatrix} 
 s_1 & s_2 & \cdots & s_{k} \\ 
 s_2 & s_3 & \cdots &  s_{k+1} \\ 
 \vdots & \vdots & \ddots &  \vdots \\ 
 s_{k-1} & s_{k} & \cdots  & s_{2k-2} \\ 
 s_{k+1} & s_{k+2} & \cdots  & s_{2k} \end{vmatrix}
 = ({\Delta}_{0,k+1} - \Delta_{2,k})\Delta_{2,k-1},
 \end{align}
 which is a consequence of Sylvester's determinant identity \cite{ba68}, \cite{ga} applied to ${\Delta}_{0,k+1} - \Delta_{2,k}$  (and using $\Delta_{1,k}=0$). 
 % More precisely, apply it to the matrix
 % \begin{align*}
 % \begin{pmatrix}
 %  s_2 & \cdots & s_k & s_0-1 & s_1 \\
 %  s_3 & \cdots & s_{k+1} & s_1 & s_2 \\
 %  \vdots & \ddots & \vdots & \vdots & \vdots \\
 %  s_{k+1} & \cdots & s_{2k-1} & s_{k-1} & s_k \\
 %  s_{k+2} & \cdots & s_{2k} & s_k & s_{k+1} 
 % \end{pmatrix}
 % \end{align*}
 Therefore, it suffices to note that the right-hand side of~\eqref{eq:deriv} does not vanish since the function~\eqref{eqnM2HN}, as well as the function 
 \begin{align*}
  - \frac{M(z)}{z} = \frac{s_0-1}{z} + \sum_{k=2}^\infty \frac{s_{k-1}}{z^k}, \quad |z|\rightarrow\infty, 
 \end{align*}
 are anti-Herglotz--Nevanlinna functions. % (cf.\ \cite[Theorem~3.4]{hoty12}). 
 
 Thus, for small enough $t\not=0$, the rational function $M_t$ has a finite continued fraction expansion of the form~\eqref{eqnContFrac} with $N_t = |\sigma|$ and coefficients $l_{t,n}$, $m_{t,n}$ given in terms of the determinants $\Delta_{t,0,k}$ and $\Delta_{t,1,k}$ for $k=0,\ldots,N_t+1$ as in \eqref{eqnCoanInv} and \eqref{eqnCobnInv}, respectively. 
 Now we define the rational functions $M_{t,n}$, $n=0,\ldots,|\sigma|$ inductively via  
 \begin{align*}
  \frac{1}{M_{t,0}(z)} = -l_{t,0}, \quad z\in\C\backslash\R, %=-\frac{\Delta_{t,1,N}^2}{\Delta_{t,0,N} \Delta_{t,0,N+1}}, \quad z\in\C\backslash\R,
 \end{align*}
 and for every $n=1,\ldots,|\sigma|$ by 
 \begin{align}\label{eqnMrec}
  \frac{1}{M_{t,n}(z)} = -l_{t,n}+\frac{1}{m_{t,n}(z)+M_{t,n-1}(z)}, \quad z\in\C\backslash\R,
  %\cfrac{\Delta_{\xi,1,N-n}^2}{\Delta_{\xi,0,N-n} \Delta_{\xi,0,N-n+1}} + \frac{1}{z \cfrac{\Delta_{\xi,0,N-n+1}^2}{\Delta_{\xi,1,N-n} %\Delta_{\xi,1,N-n+1}} + M_{\xi,n-1}(z)},
 \end{align}
 such that $M_t = 1 + M_{t,|\sigma|}$. 
 
 Since the determinants $\Delta_{t,1,|\sigma|}$ are always non-zero, the functions $M_{t,0}$ converge pointwise to the constant function $M_0$ defined by  
 \begin{align*}
  \frac{1}{M_0(z)} = - \frac{\Delta_{1,|\sigma|}^{2}}{\Delta_{0,|\sigma|} \Delta_{0,|\sigma|+1}}, \quad z\in\C\backslash\R,
 \end{align*}
 as $t\rightarrow0$.  
 Now let $n\in\lbrace 1,\ldots,|\sigma|\rbrace$, suppose that $\Delta_{1,|\sigma|-(n-1)}$ is non-zero and that the functions $M_{t,n-1}$ converge pointwise to some function $M_{n-1}$ as $t\rightarrow0$. 
 If the determinant $\Delta_{1,|\sigma|-n}$ does not vanish, then we immediately infer from~\eqref{eqnMrec} and~\eqref{eqnCobnInv} that the functions $M_{t,n}$ converge pointwise to the rational function $M_n$ given by 
 \begin{align*}
  \frac{1}{M_n(z)} = - \frac{\Delta_{1,|\sigma|-n}^{2}}{\Delta_{0,|\sigma|-n} \Delta_{0,|\sigma|-n+1}} + \cfrac{1}{z \cfrac{\Delta_{0,|\sigma|-n+1}^2}{\Delta_{1,|\sigma|-n} \Delta_{1,|\sigma|-n+1}} + M_{n-1}(z)}, \quad z\in\C\backslash\R,
 \end{align*}
 as $t\rightarrow0$.
 Otherwise, if $\Delta_{1,|\sigma|-n}$ is zero, then $n<|\sigma|$ and the sum   
 \begin{align*}
  \dot{m}_{t,n}(0)+\dot{m}_{t,n+1}(0) = \frac{\Delta_{t,0,|\sigma|-n+1}^2}{\Delta_{t,1,|\sigma|-n} \Delta_{t,1,|\sigma|-n+1}} + \frac{\Delta_{t,0,|\sigma|-n}^2}{\Delta_{t,1,|\sigma|-n-1} \Delta_{t,1,|\sigma|-n}}
 \end{align*}
 converges to the number 
 \begin{align}\label{eq:beta}
  \beta_{|\sigma|-n} = \frac{\Delta_{-1,|\sigma|-n}}{ \Delta_{1,|\sigma|-n-1}} - \frac{\Delta_{-1,|\sigma|-n+2}}{\Delta_{1,|\sigma|-n+1}},
 \end{align}
 as $t\rightarrow 0$. 
 In fact, this can be seen upon employing the determinant identity~\eqref{eqnDeltaRel}. 
 Now recall that the function $M_{t,n+1}$ satisfies  
 \begin{align*}
   \frac{1}{M_{t,n+1}(z)}=-l_{t,n+1}+\cfrac{1}{m_{t,n+1}(z)+\cfrac{1}{-l_{t,n}+\cfrac{1}{m_{t,n}(z)+M_{t,n-1}(z)}}}, \quad z\in\C\backslash\R.
 \end{align*} 
 Noting that $\Delta_{1,|\sigma|-n}=0$ as well as using \eqref{eqnCoanInv} and \eqref{eqnCobnInv}, we obtain for $z\in\C\backslash\R$ 
 \begin{align}
  l_{t,n}m_{t,n}(z) & \to 0, \nonumber\\
  l_{t,n}m_{t,n+1}(z) & \to 0,\nonumber\\
  l_{t,n}m_{t,n}(z)m_{t,n+1}(z)&\to z^2\frac{\Delta_{0,|\sigma|-n}\Delta_{0,|\sigma|-n+1}}{\Delta_{1,|\sigma|-n-1}\Delta_{1,|\sigma|-n+1}} = -\alpha_{|\sigma|-n} z^2,\label{eq:alpha}
 \end{align}
 as $t\to 0$. 
 Hereby, note that $\alpha_{|\sigma|-n}>0$ since Sylvester's determinant identity shows
 \begin{align*}
  -\Delta_{1,k-1} \Delta_{1,k+1} = \begin{vmatrix} s_1 & s_2 & \cdots & s_{k-1} & s_{k+1} \\ 
                                                      s_2 & s_3 & \cdots & s_{k} & s_{k+2} \\ 
                                                      \vdots & \vdots & \ddots & \vdots & \vdots \\ 
                                                    % s_{k-1} & s_{k} & \cdots & s_{2k-3} & s_{2k-1} \\ 
                                                      s_{k} & s_{k+1} & \cdots & s_{2k-2} & s_{2k} \end{vmatrix}^2 
\end{align*}
 as long as $\Delta_{1,k}=0$. 
%Alternativelty, this follows since $\dot{m}_{t,n}(0)$ and $\dot{m}_{t,n+1}(0)$ both blow up but the sum $\dot{m}_{t,n}(0) + \dot{m}_{t,n+1}(0)$ has a finite limit. Therefore, $\dot{m}_{t,n}(0)$ and $\dot{m}_{t,n+1}(0)$ have opposite sign for small $t\not=0$. 
 Consequently, a computation shows that the functions $M_{t,n+1}$ converge pointwise to the rational function $M_{n+1}$ given by 
 \begin{align*}%\label{eqnMn+1}
  \frac{1}{M_{n+1}(z)} = - \cfrac{\Delta_{1,|\sigma|-n-1}^2}{\Delta_{0,|\sigma|-n-1}\Delta_{0,|\sigma|-n}}  +\cfrac{1}{z^2\alpha_{|\sigma|-n}+z\beta_{|\sigma|-n}+M_{n-1}(z)}, \quad z\in\C\backslash\R,
 \end{align*}
 as $t\rightarrow0$. 
 Also note that our induction hypothesis $\Delta_{1,|\sigma|-(n+1)}\not=0$ is satisfied. 
 Concluding, one observes that the functions $M_t$ converge pointwise to our initial function $M$ as $t\rightarrow0$ and hence $M = 1 + M_{|\sigma|}$. 
 Thus, the function $M$ indeed has a finite continued fraction expansion of the form \eqref{eqnContFrac} for some integer $N\in\N_0$ (which may be less than $|\sigma|$), some reals $l_n>0$ for $n=0,\ldots,N$ and some real, non-constant polynomials $m_n$ of degree at most two with $m_n(0)=0$ and $\ddot{m}_n(0)\geq0$ for every $n=1,\ldots,N$. 
% All these quantities can be written down explicitly in terms of the moments $s_k$, $k\in\N_0$ as indicated in the proof. 
 
 Now given this continued fraction expansion of the form~\eqref{eqnContFrac} for $M$, we may define discrete measures $\omega$ and $\dip$ of the form~\eqref{eqnMEA} such that~\eqref{eqnan} and~\eqref{eqnbn} hold. 
 Consequently, the corresponding Weyl--Titchmarsh function is precisely $M$ and hence the corresponding spectrum is $\sigma$ and for $\lambda\in\sigma$ the norming constant is $\gamma_\lambda^2$.  
 
  {\em Uniqueness.} 
  From Lemma~\ref{lem:herg} one sees that the Weyl--Titchmarsh function is uniquely determined by the spectrum and the corresponding norming constants. 
  On the other side, the polynomials $m_n$, $n=1,\ldots,N$ as well as the quantities $l_n$, $n=0,\ldots,N$ can be read off from the continued fraction expansion~\eqref{eqnContFrac} of the Weyl--Titchmarsh function upon letting $|z|\rightarrow\infty$. 
  Altogether, this implies that the solution to our inverse spectral problem is unique. 
 \end{proof}

 Lemma~\ref{lemCF} and the proof of Theorem~\ref{thmIP} also show that it is possible to tell from the spectral quantities whether the measure $\dip$ is actually present or not. 

\begin{corollary}\label{cor:v=0}
 Let $\omega$ and $\dip$ be measures of the form in \eqref{eqnMEA}, $\sigma$ be the associated spectrum and for each $\lambda\in\sigma$ let $\gamma_\lambda^2$ be the corresponding norming constant. 
 \begin{enumerate}[label=(\roman*), ref=(\roman*), leftmargin=*, widest=iii] 
 \item\label{it:v=0i} We have $\dip=0$ if and only if the leading principal minors $\Delta_{1,k}$, $k=1,\ldots,|\sigma|$, of the moment matrix  
 \begin{align}\label{eqnLPM}
  \begin{pmatrix} s_1 & s_2 & \cdots & s_{|\sigma|} \\ s_2 & s_3 & \cdots & s_{|\sigma|+1} \\ \vdots & \vdots & \ddots & \vdots \\ s_{|\sigma|} & s_{|\sigma|+1} & \cdots & s_{2|\sigma|-1} \end{pmatrix}, \qquad\text{with } s_k = \sum_{\lambda\in\sigma} \frac{\lambda^{k}}{\lambda \gamma_\lambda^{2}}, \quad k\in\N,
 \end{align}
  are non-zero, where $|\sigma|$ denotes the number of eigenvalues.
  \item\label{it:v=0ii} If $\eta_+$ denotes the number of nonzero elements in $\dip_1,\ldots,\dip_N$ and $\kappa_0$ denotes the number of zero elements in $\Delta_{1,1},\ldots,\Delta_{1,|\sigma|}$, then we have the relation  
  \begin{align} 
   |\sigma| - N & = \eta_+ =\kappa_0.
  \end{align}
  \item\label{it:v=0iii} In the case $\dip=0$, the measure $\omega$ is non-negative (non-positive) if and only if all minors $\Delta_{1,k}$ are positive (their signs are alternating starting with a negative sign, that is, $s_1 < 0$). 
\end{enumerate}
\end{corollary}

\begin{remark}
Note that Corollary~\ref{cor:v=0}~\ref{it:v=0i} can be found in \cite{besasz00} (cf.\ Theorem~5.5). 
Part~\ref{it:v=0iii} is well-known and due to Stieltjes. The main novelty is part~\ref{it:v=0ii}, which describes the range of the inverse spectral map in the case when at least one of minors $\Delta_{1,k}$, $k=1,\ldots,|\sigma|$ is zero. 
\end{remark}

We complete this section by summarizing formulas that provide the solution of the inverse spectral problem for the generalized string~\eqref{eqnDE}. 
In order to do this, for each $n=0,\ldots,N$ we denote with $\kappa(n)$ the largest integer $k=0,\ldots,|\sigma|$ such that the numbers of nonzero elements of $\Delta_{1,k},\ldots,\Delta_{1,|\sigma|}$ is exactly $n+1$. 
According to this definition, the determinant $\Delta_{1,\kappa(n)}$ is precisely the $(n+1)$-th nonzero member of the sequence $\Delta_{1,1},\ldots,\Delta_{1,|\sigma|}$, counting downwards. 

\begin{corollary}\label{cor:invsol}
 Let $\omega$ and $\dip$ be measures of the form in \eqref{eqnMEA}, $\sigma$ be the associated spectrum and for each $\lambda\in\sigma$ let $\gamma_\lambda^2$ be the corresponding norming constant.
 Then we have $N = |\sigma| - \kappa_0$ and the points of support of $\omega$ and $\dip$ are given by 
\begin{align}\label{eq:x_n}
 x_{n} = \log\left(\frac{\Delta_{0,\kappa(n)+1}}{\Delta_{2,\kappa(n)}}-1\right),\quad n=1,\ldots,N.
\end{align}
 Moreover, for each $n\in\lbrace 1,\ldots,N\rbrace$ the corresponding weights are given by 
\begin{align}\label{eq:wv_n01}
 \omega_{n} & =\frac{\Delta_{2,\kappa(n)}(\Delta_{0,\kappa(n)+1}-\Delta_{2,\kappa(n)})}{\Delta_{1,\kappa(n)}\Delta_{1,\kappa(n)+1}}, & 
 \dip_{n} & =0,
\end{align}
 if $\Delta_{1,\kappa(n)+1}\not=0$ as well as by   
\begin{align}
  \omega_{n} & = \frac{\Delta_{2,\kappa(n)}(\Delta_{0,\kappa(n)+1}-\Delta_{2,\kappa(n)})}{\Delta_{0,\kappa(n)+1}^2}
\left( \frac{\Delta_{-1,\kappa(n)+1}}{\Delta_{1,\kappa(n)}} - \frac{\Delta_{-1,\kappa(n)+3}}{\Delta_{1,\kappa(n)+2}} \right),\label{eq:w_n02}\\
  \dip_{n} & = - \frac{\Delta_{0,\kappa(n)+2}}{\Delta_{0,\kappa(n)+1}} \frac{\Delta_{2,\kappa(n)}(\Delta_{0,\kappa(n)+1} - \Delta_{2,\kappa(n)})}{\Delta_{1,\kappa(n)}\Delta_{1,\kappa(n)+2}}, \label{eq:v_n02}
\end{align}
 if $\Delta_{1,\kappa(n)+1}=0$.
\end{corollary}

\begin{proof}
The proof of Theorem~\ref{thmIP} shows that the coefficients $l_n$ in the finite continued fraction expansion~\eqref{eqnContFrac} of $M$ are given by 
\begin{align*}
 l_n = \frac{\Delta_{1,\kappa(n)}^2}{\Delta_{0,\kappa(n)}\Delta_{0,\kappa(n)+1}}, \quad n=0,\ldots,N,
\end{align*}
 and furthermore for each $n=1,\ldots,N$ the polynomial $m_n$ is given by  
\begin{align*}
 m_n(z) = z \frac{\Delta_{0,\kappa(n)+1}^2}{\Delta_{1,\kappa(n)} \Delta_{1,\kappa(n)+1}}, \quad z\in\C.
\end{align*}
 if $\Delta_{1,\kappa(n)+1} \not=0$ and by 
 \begin{align*}
  m_n(z) = z\left( \frac{\Delta_{-1,\kappa(n)+1}}{\Delta_{1,\kappa(n)}} - \frac{\Delta_{-1,\kappa(n)+3}}{\Delta_{1,\kappa(n)+2}} \right) - z^2 \frac{\Delta_{0,\kappa(n)+1} \Delta_{0,\kappa(n)+2}}{\Delta_{1,\kappa(n)} \Delta_{1,\kappa(n)+2}}, \quad z\in\C. 
 \end{align*}
 if $\Delta_{1,\kappa(n)+1} = 0$.
 Moreover, with the notation from the proof of Theorem~\ref{thmIP}, we obtain for each $n\in\lbrace0,\ldots,N\rbrace$ the identity 
 \begin{align*}
  \sum_{j=n}^{N} l_j & = \sum_{j=n}^N \frac{\Delta_{1,\kappa(j)}^2}{\Delta_{0,\kappa(j)}\Delta_{0,\kappa(j)+1}} = \sum_{k=0}^{\kappa(n)} \frac{\Delta_{1,k}^2}{\Delta_{0,k}\Delta_{0,k+1}} = \lim_{t\rightarrow 0} \sum_{k=0}^{\kappa(n)} \frac{\Delta_{t,1,k}^2}{\Delta_{t,0,k}\Delta_{t,0,k+1}} \\
                     & = \lim_{t\rightarrow 0} \frac{\Delta_{t,2,\kappa(n)}}{\Delta_{t,0,\kappa(n)+1}} = \frac{\Delta_{2,\kappa(n)}}{\Delta_{0,\kappa(n)+1}}
 \end{align*}
 upon employing \cite[(II.8)]{sti} (cf.\ also \cite[(5.8)]{besasz00}) for the fourth equality. 
 As a consequence, we now obtain  
 \begin{align*}
  1-\tanh\left(\frac{x_{n}}{2}\right)=2 \sum_{j=n}^{N} l_j = 2 \frac{\Delta_{2,\kappa(n)}}{\Delta_{0,\kappa(n)+1}}, \quad n=1,\ldots,N, 
 \end{align*}
 which proves~\eqref{eq:x_n}.
 Finally, combining \eqref{eq:x_n} 
 % that is, 
 % \begin{align*}
 % \frac{1}{4 \cosh^2(x_n/2)} = \frac{\Delta_{2,\kappa(n)} (\Delta_{0,\kappa(n)+1} - \Delta_{2,\kappa(n)})}{\Delta_{0,\kappa(n)+1}^2},
 % \end{align*}
 with the formulas for the polynomials $m_n$, $n=1,\ldots,N$, yields the remaining claims. 
\end{proof}

\begin{remark}
The main results of this section are not new. In particular, Corollary~\ref{cor:invsol} was obtained by Stieltjes \cite{sti} under the assumption that all determinants $\Delta_{1,k}$, $k=1,\ldots,|\sigma|$ are positive. 
Everything remains true if these determinants are non-zero (cf.\ \cite{besasz00, hoty12, krla79, krla80}). 
In the general case, the solution of the inverse problem was given by Krein and Langer in \cite{krla79, krla80} (see also \cite{de97, lawi98}).
 \end{remark}

\section{The conservative Camassa--Holm flow}\label{secGCMPS}

 Next, we will show that our generalized spectral problem indeed serves as an isospectral problem for the conservative Camassa--Holm equation in the multi-peakon case. 
 For this purpose, let  $\omega(\,\cdot\,,t)$ and $\dip(\,\cdot\,,t)$ be discrete measures of the form in Section~\ref{s1} for each $t\in\R$. More precisely, this means that there is some integer $N(t)\in\N_0$ and strictly increasing $x_1(t),\ldots,x_{N(t)}(t)\in\R$ such that  
 \begin{align}
  \omega(\,\cdot\,,t) = \sum_{n=1}^{N(t)} \omega_n(t)\delta_{x_n(t)} \quad\text{and}\quad \dip(\,\cdot\,,t) = \sum_{n=1}^{N(t)} \dip_n(t) \delta_{x_n(t)},
 \end{align}
 where $\omega_n(t)\in\R$ and $\dip_n(t)\geq 0$  for $n=1,\ldots,N(t)$.  
 As before, for definiteness, we will also assume that $|\omega_n(t)| + \dip_n(t) > 0$ for each $n=1,\ldots,N(t)$. 
 Associated with the measures $\omega$ and $\dip$ is the function $u$ on $\R\times\R$ given by 
 \begin{align}\label{eqn4u}
  u(x,t) = \frac{1}{2} \int_\R \E^{-|x-s|} d\omega(s,t) = \frac{1}{2} \sum_{n=1}^{N(t)} \omega_n(t)\, \E^{-|x-x_n(t)|}, \quad x,\, t\in\R,
 \end{align}
 as well as the non-negative Borel measures $\mu(\,\cdot\,,t)$ on $\R$ given by 
 \begin{align}\label{eqn4mu}
  \mu(B,t) = \dip(B,t) + \int_B |u(x,t)|^2 + |u_x(x,t)|^2 dx, \quad B\in\B(\R), 
 \end{align}
 for each $t\in\R$. 
 Of course, the pair $(u,\mu)$ will play the role of a (potential) global conservative multi-peakon solution of the Camassa--Holm equation. 
 In this respect, let us mention that every global conservative multi-peakon solution is of the particular form of~\eqref{eqn4u} and~\eqref{eqn4mu}. 
 Finally, one should note that it is always possible to go back and forth between the pair $(u,\mu)$ and the measures $\omega$ and $\dip$. 
  
 Now consider the family (parametrized by time $t\in\R$) of spectral problems 
 \begin{align}\label{eqnTDspecprob}
  -f''(x) + \frac{1}{4} f(x) = z\, \omega(x,t) f(x) + z^2 \dip(x,t) f(x), \quad x\in\R,
 \end{align}
 with a complex spectral parameter $z\in\C$.  
 We will denote all spectral quantities associated with this spectral problem as in the preceding sections but with an additional time parameter. 
 In particular, the spectrum of \eqref{eqnTDspecprob} will be denoted with $\sigma(t)$ and the corresponding norming constants with $\gamma_{\lambda}^2(t)$ for each $\lambda\in\sigma(t)$.  
 The connection between these spectral problems and the conservative Camassa--Holm equation now lies in the following observation. 
 
 \begin{theorem}\label{thmTE}
  The pair $(u,\mu)$ is a global conservative multi-peakon solution of the Camassa--Holm equation if and only if the problems \eqref{eqnTDspecprob} are isospectral with  
  \begin{align}\label{eqnTE}
   \gamma_{\lambda}^2(t) = \E^{-\frac{t-t_0}{2\lambda}} \gamma_{\lambda}^2(t_0), \quad t\in\R,~ \lambda\in\sigma(t_0). 
  \end{align}
 \end{theorem}

 \begin{proof} {\em Necessity.}
  Suppose that the pair $(u,\mu)$ is a global conservative multi-peakon solution of the Camassa--Holm equation and denote with $\Gamma$ the set of all times $t\in\R$ for which $\dip(\,\cdot\,,t)\not=0$. 
  First of all, we will show that the spectral quantities have the claimed time evolution locally near some arbitrary fixed time $t_0\in\R\backslash\Gamma$. 
  Since the set $\Gamma$ consists of isolated points (as noted in the introduction), we know that $N$ is constant (to say $N_0\in\N_0$) on $\R\backslash\Gamma$ and that the measure $\omega$ evolves according to
  \begin{align}\label{eqnHSxo}
   x_n' & = \frac{1}{2} \sum_{k=1}^{N_0} \omega_k\, \E^{-|x_n-x_k|}, & \omega_n' = \frac{1}{2} \sum_{k=1}^{N_0} \omega_n \omega_k\, \sgn(x_n-x_k)\, \E^{-|x_n-x_k|}
  \end{align}
  near $t_0$.
  In particular, for each $n\in\lbrace 1,\ldots,N_0\rbrace$ the position $x_n$ and weight $\omega_n$ are continuously differentiable within the interval $(t_0-\delta,t_0+\delta)$ for some $\delta>0$. 
  Consequently, the function $\Psi$ defined by 
  \begin{align}\label{eq:Psi}
   \Psi(x,t) = \begin{pmatrix} \psi(x,t) \\ \psi_x(x,t) \end{pmatrix} = \E^{-\frac{t}{4z}} \begin{pmatrix} \phi_-(z,x,t) \\ \phi_{-}'(z,x,t) \end{pmatrix}, \quad x,\,t\in\R,
  \end{align}
  (for any fixed nonzero $z\in\C$), has continuous first partial derivatives in each of the regions $U_n = \lbrace (x,t)\in\R\times(t_0-\delta,t_0+\delta) \,|\, x_n(t) < x < x_{n+1}(t) \rbrace$, $n=0,\ldots,N_0$.  
  In fact, this is obvious in the region $U_0$ in view of~\eqref{eqnPHIAsym}. 
  For the general case, note that for two adjacent regions, say $U_{n-1}$ and $U_{n}$, the function $\Psi$ obeys  
  \begin{align*}
   \Psi(x,t) = T(x,x_n(t))\Omega(\omega_n(t)) T(x_n(t),\tilde{x}) \Psi(\tilde{x},t), \quad (x,t)\in U_{n}, ~ (\tilde{x},t)\in U_{n-1},
  \end{align*}  
  where the matrices appearing in the above equation are given by
  \begin{align*}
   T(x,\tilde{x}) & = \begin{pmatrix} \cosh\left(\frac{x-\tilde{x}}{2}\right) & 2\sinh\left(\frac{x-\tilde{x}}{2}\right) \\ \frac{1}{2} \sinh\left(\frac{x-\tilde{x}}{2}\right) & \cosh\left(\frac{x-\tilde{x}}{2}\right)  \end{pmatrix}, \quad x,\, \tilde{x}\in\R, \\
   \Omega(\alpha) & = \begin{pmatrix} 1 & 0 \\ - z \alpha & 1 \end{pmatrix}, \quad \alpha\in\R.
  \end{align*}
  Thus, the auxiliary function
  \begin{align*}
   \theta(x) = \psi_{t}(x,t_0) + \left(\frac{1}{2z} + u(x,t_0)\right)\psi_{x}(x,t_0) - \frac{1}{2} u_x(x,t_0) \psi(x,t_0)
  \end{align*}
  is well-defined for $x\in\R\backslash\lbrace x_1(t_0),\ldots,x_{N_0}(t_0)\rbrace$.
  In order to prove that this function vanishes identically, we will first show that it is a solution of~\eqref{eqnTDspecprob} at time $t=t_0$. 
  Therefore, observe that $\theta$ is smooth away from the points $\lbrace x_1(t_0),\ldots,x_{N_0}(t_0)\rbrace$ and one readily verifies that % (also using that $u_{xx}(x,t_0) = u(x,t_0)$ and $\psi''(x,t_0) = \frac{1}{4} \psi(x,t_0)$ there)
  \begin{align*}
   -\theta''(x) + \frac{1}{4}\theta(x) = 0, \quad x\in\R\backslash\lbrace x_1(t_0),\ldots,x_{N_0}(t_0)\rbrace.
  \end{align*}
  % Since $\psi(x,t) = A(t) \E^{\frac{x}{2}} + B(t) \E^{-\frac{x}{2}}$ locally, we have $\psi_t(x,t_0) = A'(t_0) \E^{\frac{x}{2}} + B'(t_0) \E^{-\frac{x}{2}}$ for $x\in\R\backslash\lbrace x_1(t_0),\ldots,x_{N_0}(t_0)\rbrace$
  % Also note that $u_{xx}(x,t_0) = u(x,t_0)$ and $\psi_{xx}(x,t_0) = \frac{1}{4} \psi(x,t_0)$ there
  Furthermore, pick some $n\in\lbrace 1,\ldots,N_0\rbrace$, some small enough $\varepsilon>0$ such that one has $(x_n(t_0)-\varepsilon,x_n(t_0)+\varepsilon) \subseteq (x_{n-1}(t_0),x_{n+1}(t_0))$ and observe that 
  \begin{align*}
   \Psi(x_n(t_0)+\varepsilon,t) & = T(x_n(t_0)+\varepsilon,x_n(t))\Omega(\omega_n(t)) \\
                                & \qquad\qquad\qquad T(x_n(t),x_n(t_0)-\varepsilon)\Psi(x_n(t_0)-\varepsilon,t),
  \end{align*}
  for $t$ near $t_0$.
  Upon differentiating with respect to $t$, evaluating at time $t=t_0$ and finally letting $\varepsilon\rightarrow0$, this gives
  \begin{align*}
   \Psi_t(x_n(t_0)+,t_0) & =  \Omega(\omega_n(t_0)) \Psi_{t}(x_n(t_0)-,t_0) \\ 
                         & \qquad\quad + z \begin{pmatrix} \omega_n(t_0) x_n'(t_0) & 0 \\ - \omega_n'(t_0) & - \omega_n(t_0) x_n'(t_0) \end{pmatrix} \Psi(x_n(t_0)-,t_0). 
  \end{align*}  
  Using this, as well as the identities 
  \begin{align*} 
   u_x(x_n(t_0)+,t_0) & = u_x(x_n(t_0)-,t_0) - \omega_n(t_0), \\
   \psi_x(x_n(t_0)+,t_0) & = \psi_x(x_n(t_0)-,t_0) - z\omega_n(t_0)\psi(x_n(t_0),t_0),
  \end{align*}
  (also note that $u(\,\cdot\,,t_0)$ and $\psi(\,\cdot\,,t_0)$ are continuous), one immediately verifies that 
  \begin{align*}
    \theta(x_n(t_0)+)  & = \theta(x_n(t_0)-) + z \omega_n(t_0) \psi(x_n(t_0),t_0) \Sigma^1_n(t_0),  \\ 
    \theta'(x_n(t_0)+) & = \theta'(x_n(t_0)-) - z \omega_n(t_0) \theta(x_n(t_0)-) \\
                       & \qquad\qquad - z\omega_n(t_0) \psi_x(x_n(t_0)-,t_0) \Sigma^1_n(t_0) - z \psi(x_n(t_0),t_0) \Sigma^2_n(t_0),
  \end{align*}
  where we used the abbreviations
  \begin{align*}
   \Sigma^1_n(t_0) & = x_n'(t_0) - u(x_n(t_0),t_0), \\ 
   \Sigma^2_n(t_0) & = \omega_n'(t_0) + \omega_n(t_0) \frac{u_x(x_n(t_0)-,t_0) + u_x(x_n(t_0)+,t_0)}{2}.
  \end{align*}
  Thus, in view of~\eqref{eqnHSxo}, the function $\theta$ indeed is a solution of~\eqref{eqnTDspecprob} at time $t=t_0$. 
  Upon noting that $\theta(x)\E^{-\frac{x}{2}}\rightarrow 0$ as $x\rightarrow -\infty$, one furthermore concludes that $\theta$ vanishes identically. 
  In particular, we obtain the equality
  \begin{align}\label{eq:psi't}
  \psi_{t}(x,t_0) =  \frac{1}{2} u_x(x,t_0) \psi(x,t_0) - \left(\frac{1}{2z} + u(x,t_0)\right)\psi_{x}(x,t_0). 
  \end{align}
  As a consequence, we may compute    
  \begin{align*}
   W_t(z,t_0) & = \E^{-\frac{x}{2}}\E^{\frac{t}{4z}} \left(\psi_{xt}(x,t_0) + \frac{1}{2}\psi_t(x,t_0) + \frac{1}{4z}\left(\psi_x(x,t_0)+\frac{1}{2}\psi(x,t_0)\right) \right) \\
              & = \E^{-\frac{x}{2}}\E^{\frac{t}{4z}} \left(u(x,t_0)+u_x(x,t_0)\right) \frac{1}{2} \left(\frac{1}{2}\psi(x,t_0) - \psi_x(x,t_0)\right)  
               =0,
  \end{align*}
  upon noting that for $x>x_{N_0}(t_0)$ one has $u_x(x,t_0)= - u(x,t_0)$.
  Thus the Wronskian, and therefore also the spectrum associated with~\eqref{eqnTDspecprob} is independent of time on each connected component of $\R\backslash\Gamma$. 
  Moreover, if $z\in\sigma(t_0)$ is an eigenvalue, then in view of~\eqref{eqnWdot} we furthermore have for all $x>x_{N_0}(t_0)$  
  \begin{align*}
   \frac{d\gamma_z^2}{dt} (t_0) = \dot{W}(z,t_0) \frac{\E^{-\frac{x}{2}}}{\E^{\frac{t_0}{4z}}\psi(x,t_0)} \left(\frac{\psi_t(x,t_0)}{\psi(x,t_0)} + \frac{1}{4z}\right) =  - \frac{1}{2z} \gamma_z^2(t_0),
  \end{align*}
  where we used (the second equality is due to the fact that $z$ is an eigenvalue) 
  \begin{align*}
   u_x(x,t_0)= - u(x,t_0) \quad\text{and}\quad \psi_x(x,t_0) = -\frac{1}{2}\psi(x,t_0),
  \end{align*} 
  for $x>x_{N_0}(t_0)$. 
  Thus, we proved that the spectral quantities have the claimed time evolution on each connected component of $\R\backslash\Gamma$. 
  
  We are left to show that our spectral quantities are continuous at each point $t^\times\in\Gamma$. 
  Therefore, introduce $x^\times_{n} = \lim_{t\rightarrow t^\times} x_n(t)$ for each $n=0,\ldots,N_0$ and note that $\omega(\,\cdot\,,t^\times)$ and $\dip(\,\cdot\,,t^\times)$ are supported on these points (cf.\ \eqref{eqnCo1} and \eqref{eqnCo2}). 
  Clearly, we have $x^\times_{1} = x_1(t^\times)$ and hence $\Psi(x_1(t)-,t) \rightarrow \Psi(x^\times_{1}-,t^\times)$ as $t\rightarrow t^\times$. 
  Now pick some $n\in\lbrace 1,\ldots,N_0-1\rbrace$, suppose that $\Psi(x_{n}(t)-,t)\rightarrow \Psi(x^\times_{n}-,t^\times)$ as $t\rightarrow t^\times$ and that $x^\times_{n}> x^\times_{n-1}$. 
  If $x^\times_{n+1}> x^\times_{n}$, then one readily sees that
  \begin{align*}
   \Psi(x_{n+1}(t)-,t) \rightarrow T(x^\times_{n+1},x^\times_{n}) \Omega\left(\lim_{t\rightarrow t^\times} \omega_n(t)\right) \Psi(x^\times_{n}-,t^\times)
  \end{align*}  
  as $t\rightarrow t^\times$. Since the limit on the right-hand side is actually $\omega(\lbrace x^\times_{n}\rbrace,t^\times)$, we see that $\Psi(x_{n+1}(t)-,t) \rightarrow \Psi(x^\times_{n+1}-,t^\times)$ as $t\rightarrow t^\times$. 
  Otherwise, if $x^\times_{n+1} = x^\times_{n}$, then 
  \begin{align*}
   \Psi(x_{n+1}(t)+,t) & = \Omega(\omega_{n+1}(t)) T(x_{n+1}(t),x_n(t)) \Omega(\omega_n(t)) \Psi(x_{n}(t)-,t)
  \end{align*}
  and a calculation, 
%  \begin{align*}
%   \Omega(\omega_{n+1}(t)) T(x_{n+1}(t),x_{n}(t)) \Omega(\omega_n(t)) = 
%   \begin{pmatrix} \cosh\frac{x_{n+1}(t) - x_{n}(t)}{2} - 2 z \omega_{n}(t) \sinh\frac{x_{n+1}(t) - x_{n}(t)}{2} & 2 \sinh\frac{x_{n+1}(t) - x_{n}(t)}{2} \\
%    \frac{1}{2} \sinh\frac{x_{n+1}(t) - x_{n}(t)}{2} - z (\omega_n(t)+\omega_{n+1}(t)) \cosh\frac{x_{n+1}(t) - x_{n}(t)}{2} + z^2 2\omega_n(t) \omega_{n+1}(t) \sinh\frac{x_{n+1}(t) - x_{n}(t)}{2} & \cosh\frac{x_{n+1}(t) - x_{n}(t)}{2} - 2z \omega_{n+1}(t) \sinh\frac{x_{n+1}(t) - x_{n}(t)}{2}  
%   \end{pmatrix}
%  \end{align*}
  taking into account~\eqref{eqnCo1} and~\eqref{eqnCo2}, 
  % this shows that $\omega_n(t) \sinh\frac{x_{n+1}(t) - x_{n}(t)}{2} \rightarrow 0$ and $\omega_{n+1}(t) \sinh\frac{x_{n+1}(t) - x_{n}(t)}{2} \rightarrow 0$ as $t\rightarrow t^\times$
  % in fact, $\left|\omega_n(t) \sinh\frac{x_{n+1}(t) - x_{n}(t)}{2}\right| \leq \frac{1}{|\omega_{n+1}(t)|} \left|\omega_n(t) \omega_{n+1}(t) \frac{x_{n+1}(t) - x_{n}(t)}{2}\right| \rightarrow 0$ by \eqref{eqnCo2} and \eqref{eqnCo1}. Similarly for the second limit...
  % Moreover, $2 z^2 \omega_n(t) \omega_{n+1}(t) \sinh\frac{x_{n+1}(t) - x_{n}(t)}{2} = z^2 \omega_n(t) \omega_{n+1}(t) (x_{n+1}(t) - x_{n}(t)) + \oo(1) \rightarrow -z^2 \dip(\lbrace x^\times_n,t^\times)$ by \eqref{eqnCo2}. 
  shows that 
  \begin{align*}
    \Omega(\omega_{n+1}(t)) T(x_{n+1}(t),x_{n}(t)) \Omega(\omega_{n}(t)) \rightarrow \begin{pmatrix} 1 & 0 \\ - z \omega(\lbrace x^\times_{n}\rbrace, t^\times) - z^2 \dip(\lbrace x^\times_{n}\rbrace, t^\times) & 1 \end{pmatrix}
  \end{align*} 
  and hence also $\Psi(x_{n+1}(t)+,t)\rightarrow \Psi(x^\times_{n+1}+,t^\times)$ as $t\rightarrow t^\times$. 
  Furthermore, in the case when $n<N_0-1$, then we clearly obtain our induction hypothesis for $n+2$, that is, $\Psi(x_{n+2}(t)-,t) \rightarrow \Psi(x^\times_{n+2}-,t^\times)$ as $t\rightarrow t^\times$ and $x^\times_{n+2} > x^\times_{n+1}$.
  Either way, after finitely many steps one ends up with $\Psi(x_{N_0}(t)+,t) \rightarrow \Psi(x^\times_{N_0}+,t^\times)$ as $t\rightarrow t^\times$ since $x^\times_{N_0} = x_{N_0}(t^\times)$.
  Consequently, one also has $\Psi(x,t) \rightarrow \Psi(x,t^\times)$ as $t\rightarrow t^\times$ for all $x> x_{N_0}(t^\times)$. 
  But this immediately implies that $W(z,t)\rightarrow W(z,t^\times)$ as $t\rightarrow t^\times$ in view of~\eqref{eqnWronski} and hence the problems~\eqref{eqnTDspecprob} are isospectral. 
  Furthermore, if $z\in\sigma(t^\times)$ is an eigenvalue, then we may conclude from~\eqref{eqnCC} and~\eqref{eqnWdot} that $\gamma_z^2(t)\rightarrow \gamma_z^2(t^\times)$ as $t\rightarrow t^\times$, which finally proves the claimed time evolution~\eqref{eqnTE}.  
   
  {\em Sufficiency.} For the converse, pick some $t_0\in\R$ such that $\mu(\,\cdot\,,t_0)$ is absolutely continuous (this is possible because of Corollary~\ref{cor:v=0}).
  By \cite[Theorem~3.4]{hora07b} there is a unique global conservative multi-peakon solution $(\tilde{u},\tilde{\mu})$ with the initial values $\tilde{u}(\,\cdot\,,t_0) = u(\,\cdot\,,t_0)$ and $\tilde{\mu}(\,\cdot\,,t_0)= \mu(\,\cdot\,,t_0)$. 
  Now the proof of the necessity and our assumption show that the corresponding spectral quantities are the same for all times $t\in\R$.
  Thus, from the uniqueness part of the inverse problem in Theorem~\ref{thmIP}, we conclude that $\tilde{u} = u$ and $\tilde{\mu} = \mu$, which finishes the proof.
 \end{proof}
 
 \begin{remark}
 For classical multi-peakon solutions, the time evolution of the spectral data under the Camassa--Holm flow was found in \cite{besasz00} and it is, of course, given by~\eqref{eqnTE}. 
 For extending this fact to global conservative multi-peakon solutions, it suffices to show that the spectral data depend continuously on time (note that such a solution is a classical multi-peakon solution away from the isolated times of blow-up). 
 More precisely, this means that one has to verify that the eigenvalues and corresponding norming constants are continuous at the times of blow-up. 
 However, this is a consequence of comparing the formulas \eqref{eqnCo1} and \eqref{eqnCo2} with the corresponding formulas \eqref{eq:beta} and \eqref{eq:alpha} in the proof of Theorem~\ref{thmIP}. 
 \end{remark}
 
 The time evolution of the spectral quantities in Theorem~\ref{thmTE} provides us with conserved quantities for global conservative multi-peakon solutions.
    
 \begin{corollary}\label{cor:5.2}
  If the pair $(u,\mu)$ is a global conservative multi-peakon solution of the Camassa--Holm equation, then the integrals
 \begin{align}\label{eqnCQ}
  \cI_1 =  \int_\R d\omega(x,t) \quad \text{and}\quad \cI_2 = \int_\R u(x,t)d\omega(x,t) + \int_\R d\dip(x,t)
 \end{align}
 are independent of time $t\in\R$.
 \end{corollary}

 \begin{proof}
  In view of the two trace formulas depicted in Proposition~\ref{proptraceformulas}, the quantities $\cI_1$ and $\cI_2$ in~\eqref{eqnCQ} are given in terms of the eigenvalues of~\eqref{eqnTDspecprob}.  
  Since the problems~\eqref{eqnTDspecprob} are isospectral by Theorem~\ref{thmTE}, this establishes the claim. 
 \end{proof}

 Since a simple integration by parts yields
  \begin{align*}
  \int_\R |u(x,t)|^2 + |u_x(x,t)|^2 dx = \int_\R u(x,t)d\omega(x,t), \quad t\in\R,
 \end{align*}
 the second integral $\cI_2$ is actually equal to $\mu(\R,t)$ (cf.\  \cite[Theorem~2.3]{hora07b}). 
 In particular, the $H^1(\R)$ norm of $u(\,\cdot\,,t)$ is conserved except for times when peaks collide.
 
 Finally, let us mention that one may investigate the (number of) times of collision by analyzing the zeros of the functions $t\mapsto \Delta_{1,k}(t)$, $k\in \{1,\dots, |\sigma|\}$, in view of Corollary~\ref{cor:v=0}. 
 In particular, we obtain the following estimate.

 \begin{proposition}\label{prop:finite}
  If the pair $(u,\mu)$ is a global conservative multi-peakon solution, then there are only finitely many times of collision, that is, the set $\Gamma$ is finite.
 \end{proposition}

 \begin{proof}
  First of all, note that  (cf.\ \cite[Lemma 6.1]{besasz00}) for $k\in \{1,\dots,|\sigma|\}$
 \begin{align*} 
  \Delta_{1,k}(t)= \sum_{J\subseteq \sigma,~ |J|=k} \frac{\Lambda_J}{\Gamma_J}\, \E^{(t-t_0)\Sigma_J},\quad t\in\R,
 \end{align*}
 where we introduced the quantities 
 \begin{align*} 
  \Lambda_J = \prod_{\lambda,\kappa\in J,~ \lambda<\kappa} (\lambda-\kappa)^2, \quad \Gamma_J=\prod_{\lambda\in J}\gamma_\lambda^{2}(t_0) \quad\text{and}\quad \Sigma_J = \sum_{\lambda\in J} \frac{1}{2\lambda}.
 \end{align*}
 Therefore, we have
 \begin{align*}
  \Delta_{1,k}(t)=\frac{\Lambda_{J^\pm}}{\Gamma_{J^\pm}} \E^{(t-t_0)\Sigma_{J^\pm}}(1+\oo(t)),\quad t\to \pm \infty,
 \end{align*}
 where the index sets $J^\pm\subseteq\sigma$ with $|J^\pm|=k$ are chosen such that 
 \begin{align*}
  \Sigma_{J^+}= \max_{J\subseteq\sigma,~ |J|=k}\Sigma_J \quad\text{and}\quad \Sigma_{J^-} = \min_{J\subseteq\sigma,~|J|=k} \Sigma_J.
 \end{align*}
 In particular, this shows that the (analytic) function $t\mapsto\Delta_{1,k}(t)$ has no zeros for large enough $|t|$. 
 Thus, this function has only finitely many zeros and hence the claim follows in view of Corollary~\ref{cor:v=0}.
 \end{proof}
    
 As a consequence of Proposition~\ref{prop:finite}, one sees that a global conservative solution of the Camassa--Holm equation is a classical multi-peakon solution after long enough time. 
 In particular, long-time asymptotics for this solution are now readily deduced from the ones given in \cite[Section~6]{besasz00}.

 \appendix 
 
 \section{Global conservative two-peakon solutions}\label{App}
 
 In this appendix, we will employ our findings in order to solve the conservative Camassa--Holm equation in the two-peakon case. 
 Therefore, let the pair $(u,\mu)$ be a global conservative solution of the Camassa--Holm equation such that for some $t_0\in\R$ the measure $\mu(\,\cdot\,,t_0)$ is absolutely continuous and 
 \begin{align}\label{eq:app01}
  u(x,t_0) = \frac{1}{2} \sum_{n=1}^2 \omega_n(t_0)\, \E^{-|x-x_n(t_0)|}, \quad x\in\R,
 \end{align} 
 where $\omega_1(t_0)$, $\omega_2(t_0)\neq 0$ and $ x_1(t_0)<x_2(t_0)$. 
 The corresponding spectral problem at time $t_0$ is given by 
 \begin{align}\label{eq:SP+-}
   -f''(x)+\frac{1}{4}f(x) = z \sum_{n=1}^2 \omega_n(t_0)\delta_{x_n(t_0)}(x) \,f(x),\quad x\in\R,
 \end{align}
 with a complex spectral parameter $z\in\C$. 
 From Corollary~\ref{cor:v=0}~\ref{it:v=0ii}, we know that there are precisely two eigenvalues, say $\lambda_1$ and $\lambda_2$ with $\lambda_1 < \lambda_2$. 
 As a consequence, the Wronskian $W$ is a polynomial of degree two which can be written down as   
 \begin{align} \label{eq:wro}
  W(z) & = 1-(\omega_1(t_0)+\omega_2(t_0))z+ \omega_1(t_0)\omega_2(t_0) \left(1-\E^{-|x_2(t_0)-x_1(t_0)|}\right) \frac{z^2}{2}, \quad z\in\C,
 \end{align}
 in view of the trace formulas~\eqref{eqnTF} in Proposition~\ref{proptraceformulas}.
 Thus, we can compute the spectrum $\sigma=\lbrace\lambda_1,\lambda_2\rbrace$ of \eqref{eq:SP+-} by solving for the roots of $W$. 
 Moreover, for every $\lambda\in\sigma$ the corresponding norming constant at time $t_0$ is given by 
 \begin{align}
  \gamma_\lambda^2(t_0) = \omega_1(t_0) \E^{-x_1(t_0)} \left( 1 - \lambda\omega_2(t_0)\left(1-\E^{x_1(t_0)-x_2(t_0)}\right) \right)^2 + \omega_2(t_0) \E^{-x_2(t_0)}, 
 \end{align}
 which may be easily obtained upon employing the integral equation~\eqref{eqnIEpm}.

 Now the time evolution of the spectral data (that is, the norming constants) under the conservative Camassa--Holm flow is given by Theorem~\ref{thmTE}:  
 \begin{align}
  \gamma_\lambda^2(t)=\gamma_\lambda^2(t_0) \E^{-\frac{t-t_0}{2\lambda}},\quad t\in\R,~\lambda\in\lbrace\lambda_1,\lambda_2\rbrace.
 \end{align} 
 Hence we may apply the solution of the inverse spectral problem in Corollary~\ref{cor:invsol} in order to find the global conservative two-peakon solution %of the Camassa--Holm equation 
 with initial data \eqref{eq:app01} at some given time $t\in\R$. 
 Therefore, one first has to compute the four moments
 \begin{align} 
  s_0(t) & =1+ \frac{1}{\lambda_1\gamma_{\lambda_1}^2(t)} + \frac{1}{\lambda_2\gamma_{\lambda_2}^2(t)}, & s_1(t) & = \frac{1}{\gamma_{\lambda_1}^2(t)} + \frac{1}{\gamma_{\lambda_2}^2(t)}, \\  
  s_2(t) & = \frac{\lambda_1}{\gamma_{\lambda_1}^2(t)} + \frac{\lambda_2}{\gamma_{\lambda_2}^2(t)}, & s_3(t) & = \frac{\lambda_1^2}{\gamma_{\lambda_1}^2(t)} + \frac{\lambda_2^2}{\gamma_{\lambda_2}^2(t)}. 
 \end{align} 
% as well as the determinants 
% \begin{align*}
%  \Delta_{1}(t) & =s_0(t)-1, & \Delta_{2}(t) & =\Delta_{0,2}(t) %s_0(t)s_2(t)-s_1(t)^2
% -s_2(t)=\frac{(\lambda_2-\lambda_1)^2}{\lambda_1\lambda_2\gamma_1^2(t)\gamma_2^2(t)}.
% \end{align*}
 In particular, from Corollary~\ref{cor:invsol} we immediately infer that 
 \begin{align}
  N(t) = \begin{cases}
          2, & s_1(t) \not= 0, \\
          1, & s_1(t) =0, 
         \end{cases}
 \end{align}
 since the determinant $\Delta_{1,1}(t)$ is simply given by $s_1(t)$. 
 
 In the first case (when $s_1(t) \not=0$), we recover the characteristics to be 
 \begin{align} \label{eqnPos}
  x_1(t) & %= \log\left( \frac{s_0(t)s_2(t) - s_1(t)^2}{s_2(t)}- 1\right) 
            = \log\left(\frac{(\lambda_2 - \lambda_1)^2}{\lambda_1\lambda_2(\lambda_2\gamma_1^2(t) + \lambda_1\gamma_2^2(t))}\right), \\
  \label{eqnPosb}
  x_2(t) & %= \log\left(s_0(t)-1\right) 
            = \log\left( \frac{1}{\lambda_1 \gamma_{\lambda_1}^2(t)}+\frac{1}{\lambda_2\gamma_{\lambda_2}^2(t)} \right).
 \end{align}
% In particular, their difference is given by 
% \begin{align*}
%  x_2(t)-x_1(t) = \log\left(\frac{(s_0(t)-1)s_2(t)}{(s_0(t)-1)s_2(t)-s_1(t)^2}\right).
% \end{align*}
 Moreover, the corresponding weights are readily written down as 
 \begin{align}\label{eq:app10}
  \omega_1(t) & =\frac{s_2(t)((s_0(t)-1)s_2(t)-s_1(t)^2)}{s_1(t)(s_1(t)s_3(t)-s_2(t)^2)}, & \omega_2(t) & =\frac{s_0(t)-1}{s_1(t)},
 \end{align}
 in view of Corollary~\ref{cor:invsol}.
 Hereby, one should also notice that 
 \begin{align}
  \omega_1(t) = \omega_1(t_0) + \omega_2(t_0) -\omega_2(t) = \frac{1}{\lambda_1} + \frac{1}{\lambda_2} + \frac{1-s_0(t)}{s_1(t)}, 
 \end{align}
 upon employing the first trace formula in~\eqref{eqnTF}. 
 
 In the second case (when $s_1(t)=0$), the characteristics coincide at the point  
 \begin{align}\label{eqnPos2}
  x_1(t)=\log\left(\frac{1}{\lambda_1\gamma_1^2(t)}+\frac{1}{\lambda_2\gamma_2^2(t)}\right).
 \end{align}
 The corresponding weight $\omega_1(t)$ and dipole $\dip_1(t)$ are then simply given by   
 \begin{align}\label{eqnOmDip}
  \omega_1(t)  % = \omega_1(t_0) + \omega_2(t_0) 
                  = \frac{1}{\lambda_1} + \frac{1}{\lambda_2}, \qquad 
  \dip_1(t)    % = - \frac{1}{2} \omega_1(t_0) \omega_2(t_0) \left(1-\E^{-|x_2(t_0)-x_1(t_0)|}\right) 
                  = - \frac{1}{\lambda_1\lambda_2},
 \end{align}
 which is most easily obtained by comparing the (time independent) Wronskian of our spectral problem at time $t$  
% which in this case is given by 
% \begin{align*}
%  W(z) = 1 - z \omega_1(t) - z^2 \dip_1(t), \quad z\in\C
% \end{align*}
 with the polynomial given in~\eqref{eq:wro}. 
 
 Summarizing, we are now able to write down the global conservative two-peakon solution $(u,\mu)$ of the Camassa--Holm equation with initial data given by~\eqref{eq:app01}. 
 Therefore, we distinguish the following two different cases:

{\em (i) The peakon--peakon case.} It follows from the representation~\eqref{eq:wro} of the Wronskian that  $\lambda_1\lambda_2>0$ if and only if $\omega_1(t_0) \omega_2(t_0) >0$. 
 In this case, the weak solution is unique ($x_2(t)>x_1(t)$ for all $t\in\R$ since $s_1(t)\neq 0$) and hence the conservative solution coincides with the classical one and is given by 
 \begin{align}
  u(x,t)=\frac{1}{2}\sum_{n=1}^2 \omega_n(t)\, \E^{-|x-x_n(t)|}, \quad x,\, t\in\R.  %+\frac{1}{2}\omega_2(t)\E^{-|x-x_2(t)|}, t\in\R.
 \end{align}
 The coefficients appearing here are given by the equations~\eqref{eqnPos}, \eqref{eqnPosb} and~\eqref{eq:app10}.
 Of course, the measures $\mu(\,\cdot\,,t)$ are absolutely continuous for all $t\in\R$.

{\em (ii) The peakon--antipeakon case.}
 Now let $\omega_1(t_0)$ and $\omega_2(t_0)$ be of different signs, that is, $\omega_1(t_0)\omega_2(t_0) < 0$. 
 In this case, there is precisely one time $t^\times\in\R$ such that $s_1(t^\times)=0$ which is given by   
 \begin{align}\label{eq:t_0}
  t^\times = t_0 + \frac{2\lambda_1\lambda_2}{\lambda_2-\lambda_1}\log\left(-\frac{\gamma_1^2(t_0)}{\gamma_2^2(t_0)}\right).
 \end{align}
 In view of the considerations above, the global conservative solution is given by 
 \begin{align}
  u(x,t)=\begin{cases}
   \frac{1}{2} \sum_{n=1}^2 \omega_n(t)\, \E^{-|x-x_n(t)|}, & t\neq t^\times,\\
   \frac{1}{2} \omega_1(t)\, \E^{-|x-x_1(t)|}, & t=t^\times.
  \end{cases}
 \end{align}
 Hereby, the coefficients are given by  the equations~\eqref{eqnPos}, \eqref{eqnPosb} and~\eqref{eq:app10} in the case $t\not=t^\times$ and by the equations~\eqref{eqnPos2} and~\eqref{eqnOmDip} in the case $t=t^\times$. 
 In order to finish our description of global conservative two-peakon solutions, we notice that the measures $\mu(\,\cdot\,,t)$ are absolutely continuous as long as $t\not=t^\times$. 
 Only the measure $\mu(\,\cdot\,,t^\times)$ admits a singular part which is supported on the point of collision of the characteristics. 
 More precisely, we have    
 \begin{align}
  \mu_{\text{s}}(B,t^\times) & = %\int_{B} |u(x,t^\times)|^2+|u_x(x,t^\times)|^2 dx + 
                                  \dip_1(t^\times) \delta_{x_1(t^\times)}(B), \quad B\in\B(\R),
 \end{align}
 where $\dip_1(t^\times)$ and $x_1(t^\times)$ are given by the equations~\eqref{eqnPos2} and~\eqref{eqnOmDip}.

\bigskip
\noindent
{\bf Acknowledgments.}
 We thank Vladimir Derkach, Katrin Grunert, Mark Malamud, Gerald Teschl and Harald Woracek for helpful discussions and hints with respect to the literature.  
 J.E.\ gratefully acknowledges the kind hospitality of the {\em Institut Mittag-Leffler} (Djursholm, Sweden) during the scientific program on {\em Inverse Problems and Applications} in spring 2013, where parts of this article were written.

\end{document}